\documentclass{amsart}[11pt]
\usepackage{amsmath,amssymb,amsthm}
\usepackage{geometry}
\usepackage{color}
\usepackage{mathtools}

\usepackage{tikz}
\usetikzlibrary{shapes.geometric, arrows}
\usetikzlibrary{decorations.pathmorphing}
\usetikzlibrary{cd}
\tikzstyle{var} = [rectangle, minimum width=1cm, minimum height=0.5cm, text centered, draw=black, fill=white!30]
\tikzstyle{output} = [rectangle, text centered, draw=black, fill=gray!20]
\tikzstyle{input} = [rectangle, text centered, draw=black, fill=green!20]
\tikzstyle{arrow} = [thick,->,>=stealth]
\tikzset{
  closed/.style = {decoration = {markings, mark = at position 0.5 with { \node[transform shape, xscale = .8, yscale=.4] {/}; } }, postaction = {decorate} },
  open/.style = {decoration = {markings, mark = at position 0.5 with { \node[transform shape, scale = .7] {$\circ$}; } }, postaction = {decorate} }
}

\newcommand{\Gr}{\text{Gr}}
\newcommand{\GW}{\text{GW}}

\newcommand{\Hom}{\text{Hom}}

\newcommand{\im}{\text{im}}

\newcommand{\Proj}{\text{Proj}}

\newcommand{\Sch}{\text{Sch}}
\newcommand{\sgn}{\text{sgn}}

\newcommand{\spn}{\text{span}}

\newcommand{\Spec}{\text{Spec}\hspace{0.2em}}

\newcommand{\Wr}{\text{Wr}}

\newcommand{\xto}[1]{\xrightarrow{#1}}

\newcommand{\hookto}{\xhookrightarrow{}}

\newcommand{\tto}{\twoheadrightarrow}

\makeatletter
\newcommand{\superimpose}[2]{%
  {\ooalign{$#1\@firstoftwo#2$\cr\hfil$#1\@secondoftwo#2$\hfil\cr}}}
\makeatother
\newcommand{\smallslash}{\mbox{\tiny/}}

\newcommand{\clhook}{\mathrel{\raisebox{0.1em}{$\mathrel{\mathpalette\superimpose{{\hspace{0.1cm}\vspace{0.1em}\smallslash}{\hookrightarrow}}}$}}}

\newcommand{\gw}[1]{\left\langle #1 \right\rangle}
\renewcommand{\part}[2]{\frac{\partial #1}{\partial #2}}

\newcommand{\A}{\mathbb{A}}
\newcommand{\C}{\mathbb{C}}

\renewcommand{\P}{\mathbb{P}}

\newcommand{\V}{\mathbb{V}}

\renewcommand{\O}{\mathcal{O}}

\let\emptyset\varnothing

\let\sec\S
\let\phi\varphi

\usepackage{mathtools}
\usepackage{longtable}
\usepackage{amstext} 
\usepackage{array}   
\usepackage{tabu}
\usepackage{diagbox}
\usepackage{mathrsfs}
\usepackage{caption}
\usepackage{subcaption}
\usepackage{ytableau}
\usepackage{microtype}
\usepackage{float}
\usepackage{mdframed}

\usepackage[fulladjust]{marginnote}
\reversemarginpar

\usepackage{xcolor}
\definecolor{darkgreen}{rgb}{0,0.30,0} 
\definecolor{darkred}{rgb}{0.75,0,0}
\definecolor{darkblue}{rgb}{0,0,0.6} 
\definecolor{lightblue}{RGB}{179, 230, 255}
\usepackage{amsthm,thmtools}
\usepackage[pdfpagelabels,pdfborder=0,colorlinks,citecolor=darkgreen,linkcolor=darkgreen,urlcolor=darkblue]{hyperref}
\usepackage{cleveref}
    
\def\makeautorefname#1#2{\expandafter\def\csname#1autorefname\endcsname{#2}}
\makeautorefname{eqn}{Equation}%
\makeautorefname{sec}{Section}%
\makeautorefname{subsec}{Subsection}%
\makeautorefname{footnote}{footnote}%
\makeautorefname{item}{item}%
\makeautorefname{figure}{Figure}%
\makeautorefname{table}{Table}%
\makeautorefname{part}{Part}%
\makeautorefname{app}{Appendix}%
\makeautorefname{cla}{claim}%
\makeautorefname{conj}{conjecture}%
\makeautorefname{cor}{corollary}%
\makeautorefname{cex}{counterexample}%
\makeautorefname{cexs}{counterexamples}%
\makeautorefname{dig}{digression}%
\makeautorefname{disc}{discussion}%
\makeautorefname{def}{definition}%
\makeautorefname{ex}{example}%
\makeautorefname{exs}{examples}%
\makeautorefname{fac}{fact}%
\makeautorefname{intu}{intuition}%
\makeautorefname{lem}{lemma}%
\makeautorefname{nota}{notation}%
\makeautorefname{note}{note}%
\makeautorefname{prop}{proposition}%
\makeautorefname{rmk}{remark}%
\makeautorefname{term}{terminology}%
\makeautorefname{thm}{theorem}%
\makeautorefname{upsh}{upshot}%
\theoremstyle{definition}
\newtheorem{theorem}{Theorem}[section]

\newtheorem{corollary}[theorem]{Corollary}

\newtheorem{definition}[theorem]{Definition}

\newtheorem{example}[theorem]{Example}

\newtheorem{lemma}[theorem]{Lemma}
\newtheorem{notation}[theorem]{Notation}

\newtheorem{proposition}[theorem]{Proposition}
\newtheorem{remark}[theorem]{Remark}

\newtheorem{realitycheck}[theorem]{Reality check}

\newtheorem{bigtheorem}{Theorem}

\makeatletter
\let\c@corollary=\c@theorem
\let\c@proposition=\c@theorem
\let\c@lemma=\c@theorem
\let\c@conjecture=\c@theorem
\let\c@definition=\c@theorem
\let\c@example=\c@theorem
\let\c@remark=\c@theorem
\let\c@notation=\c@theorem
\let\c@equation\c@theorem
\makeatother

\usepackage[color=lightblue,textsize=tiny]{todonotes}

\usepackage{fancyhdr}
\let\sec\S
\newcommand{\til}[1]{\widetilde{#1}}
\providecommand{\ind}{\text{ind}}
\newcommand{\Pl}{\text{Pl}}
\renewcommand{\O}{\mathcal{O}}
\newcommand{\ev}{\text{ev}}

\let\del\partial

\providecommand{\HOM}{\mathcal{H}\hspace{-0.1em}\textit{om}}

\providecommand{\tr}{\text{tr}}

\title{An Enriched Degree of the Wronski}
\author{Thomas Brazelton}
\date{Last Compiled: \today}

\usepackage[style=alphabetic,maxnames=10,backend=biber]{biblatex}
\bibliography{wr}

\begin{document}

\begin{abstract} Given $mp$ different $p$-planes in general position in $(m+p)$-dimensional space, a classical problem is to ask how many $p$-planes intersect all of them. For example when $m=p=2$, this is precisely the question of ``lines meeting four lines in 3-space'' after projectivizing. The Brouwer degree of the Wronski map provides an answer to this general question, first computed by Schubert over the complex numbers and Eremenko and Gabrielov over the reals. We provide an enriched degree of the Wronski for all $m$ and $p$ even, valued in the Grothendieck--Witt ring of a field, using machinery from $\A^1$-homotopy theory. We further demonstrate in all parities that the local contribution of an $m$-plane is a determinantal relationship between certain Pl\"{u}cker coordinates of the $p$-planes it intersects.

\end{abstract}

\maketitle

\section{Introduction}
Given $m$ functions $f_1(t), \ldots, f_m(t)$ of maximum degree equal to $m+p-1$, we define the \textit{Wronski}
\begin{align*}
    \Wr(f_1, \ldots, f_m)(t) := \begin{vmatrix} f_1(t) & f_2(t) & \cdots & f_m(t) \\
                           f_1'(t) & f_2'(t) & \cdots & f_m'(t) \\
                          \vdots & \vdots & \ddots & \vdots \\
                         f_1^{(m-1)}(t) & f_2^{(m-1)}(t) & \cdots & f_m^{(m-1)}(t) \end{vmatrix}.
\end{align*}
This is a polynomial of degree at most $mp$. Let $k_{m+p-1}[t]$ denote the vector space of polynomials of degree at most $m+p-1$ over a field $k$. We observe that if $s$ is a root of the Wronski polynomial, then the $m$-plane $\spn \left\{ f_1, \ldots, f_m \right\} \subseteq k_{m+p-t}[t]$ intersects the $p$-plane $E_p(s) = \spn\left\{ (t-s)^{m+p-1}, \ldots, (t-s)^m \right\}$ nontrivially. Thus the fiber of the Wronski counts certain $m$-planes intersecting $mp$ different $p$-planes.

We could also envision these polynomials $f_i$ as defining a rational curve by $\P^1 \to \P^{m-1}$, given by $t \mapsto \left[ f_1(t): \ldots : f_m(t) \right]$. In this case $s$ is a root of the Wronski if and only if the vectors $\phi(s), \phi'(s),\ldots, \phi^{(m-1)}(s)$ do not span all of $\P^{m-1}$. We say that $\phi$ \textit{inflects} at such a point. Thus the fiber of the Wronski counts rational curves of degree $(m+p-1)$ with $mp$ prescribed inflection points. Viewing the polynomials $f_i$ as spanning an $m$-plane in the $(m+p)$-dimensional vector space of polynomials over $k$ of degree at most $(m+p-1)$, we can consider the Wronski as a map of $mp$-dimensional varieties
\begin{equation}\label{eqn:Wr}
\begin{aligned}
    \Wr: \Gr_k(m,m+p) \to \P_k^{mp} = \Proj (k_{mp}[t]).
\end{aligned}
\end{equation}
In 1886, Schubert \cite{Schubert} formulated the number of $m$-planes meeting $mp$ general $p$-planes in $(m+p)$-dimensional space as
\begin{equation}\label{eqn:nC}
\begin{aligned}
    n_\mathbb{C} = \frac{1! 2! \cdots (p-1)! (mp)!}{m!(m+1)! \cdots (m+p-1)!}.
\end{aligned}
\end{equation}
This admits a combinatorial description in that it counts the number of \textit{standard Young tableaux} of size $m \times p$. It is also the Brouwer degree of the complex Wronski (\autoref{eqn:Wr} when $k= \mathbb{C}$). More than a century later, Eremenko and Gabrielov computed the Brouwer degree of the real Wronski (\autoref{eqn:Wr} when $k= \mathbb{R}$) \cite{EG-deg-2,EG}, which also admits a combinatorial description, being the number of \textit{semi-shifted} standard Young tableaux of size $m \times p$ \cite{HoffmanHumphreys,White}.
\begin{align*}
    n_\mathbb{R} &=  \begin{cases} \frac{1! 2! \cdots (p-1)! (m-1)! (m-2)! \cdots (m-p+1)! (mp/2)!}{(m-p+2)! (m-p+4)! \cdots (m+p-2)! \left( \frac{m-p+1}{2} \right)! \left( \frac{m-p+3}{2} \right)! \cdots \left(\frac{m+p-1}{2}\right)!} & m+p \text{ odd} \\ 0 & m+p \text{ even}. \end{cases}
\end{align*}
We attach the first few values of these for the reader's reference:
\begin{figure}[H]
\centering
\begin{minipage}{.5\textwidth}
  \centering
  \begin{tabular}{|l | l l l l |}
    \hline
    \diagbox{$p$}{$m$} & $2$ & $3$ & $4$ & $5$ \\
    \hline
    $2$ & $2$ & $5$ & $14$ & $42$ \\
    $3$ & $5$ & $42$ & $462$ & $6006$ \\
    $4$ & $14$ & $462$ & $24024$ & $1662804$\\
    $5$ & $42$ & $6006$ & $1662804$ & $701149020$ \\
    \hline
    \end{tabular}
    \caption{First few values of $n_\mathbb{C}$}\end{minipage}%
\begin{minipage}{.5\textwidth}
  \centering
  \begin{tabular}{|l | l l l l |}
    \hline
    \diagbox{$p$}{$m$} & $2$ & $3$ & $4$ & $5$ \\
    \hline
    $2$ & $0$ & $1$ & $0$ & $2$ \\
    $3$ & $1$ & $0$ & $2$ & $0$ \\
    $4$ & $0$ & $2$ & $0$ & $12$\\
    $5$ & $2$ & $0$ & $12$ & $0$ \\
    \hline
    \end{tabular}
    \caption{First few values of $n_\mathbb{R}$}
\end{minipage}
\end{figure}
In this paper we unify these two computations into a single enriched Brouwer degree in the case when $m$ and $p$ are both even. The algebrao-geometric analogue of the Brouwer degree that we use is called the $\A^1$\textit{-Brouwer degree}, first defined by Morel \cite{Morel-ICM}, which is valued in the Grothendieck--Witt group of symmetric bilinear forms over $k$. This tool has been instrumental in the development of $\A^1$\textit{-enumerative geometry} (or \textit{enriched enumerative geometry}). This program has grown in recent years due to seminal work of Levine \cite{Levine}, Kass and Wickelgren \cite{KW-EKL}, Bachmann and Wickelgren \cite{BW3}, among others.

\begin{bigtheorem}\label{thm:main-thm-deg-wr} (As~\autoref{thm:deg-wronski-both-even}) Let $k$ be any field in which $(m+p-1)!$ is invertible, and let $m$ and $p$ both be even. Then the $\A^1$-degree of the Wronski $\Wr:\Gr_k(m,m+p) \to \P_k^{mp}$ is
\begin{align*}
    \deg^{\A^1} \Wr = \frac{n_\C}{2} \mathbb{H},
\end{align*}
where $n_\C$ is the Brouwer degree of the Wronski over the complex numbers, and $\mathbb{H}$ denotes the hyperbolic form $\left\langle 1,-1 \right\rangle$.
\end{bigtheorem}
Given a closed point $W = \spn \left\{ f_1, \ldots, f_m \right\}$ with Wronski polynomial having roots at distinct scalars $s_1, \ldots, s_{mp} \in k$, we may compute the local degree of the Wronski in any parities.
\begin{bigtheorem} (As~\autoref{thm:formula-local-index}) Let $k$ be any field in which $(m+p-1)!$ is invertible, and let $W \in \Gr_k(m,m+p)$ be a closed point whose Wronski polynomial is of the form $\Wr(W)(t) = \prod_{i=1}^{mp}(t-s_i)$ for distinct $s_i \in k$. Then we have that
\begin{align*}
    \deg_W^{\A^1}(\Wr) = \left\langle C \cdot \det \mathcal{B} \right\rangle,
\end{align*}
where $C$ is a fixed constant depending only on $m$, $p$, and the $s_i$'s, and $\mathcal{B}$ is a matrix of distinguished Pl\"{u}cker coordinates of the $p$-planes $E_p(s_1), \ldots, E_p(s_{mp})$.
\end{bigtheorem}
The $\ell$th column of $\mathcal{B}$ consists of $mp$ distinguished Pl\"{u}cker coordinates of the plane $E_p(s_\ell)$, and each row corresponds to the same coordinate. Thus considering the columns as vectors over $k$, we have that $\det \mathcal{B}$ is a signed volume of vectors determined by the $p$-planes that $\spn \left\{ f_1, \ldots, f_m \right\}$ intersects.

As the Wronski also counts rational curves with prescribed inflection data, we provide evidence that the local $\A^1$-degree encodes information about the geometry of the associated rational curve. In \autoref{cor:quartics} we demonstrate that the local degree at a planar quartic aligns with an enriched \textit{Welschinger invariant} in the sense of \cite{KLSW}.

\subsection{Outline}

In \autoref{sec:preliminaries}, we provide some historical background for studying the Brouwer degree of the Wronski, before exploring in greater detail the technical machinery. We discuss the rational normal curve, Grassmann duality, and Pl\"{u}cker coordinates, before providing relevant background from $\A^1$-enumerative geometry. We discuss relative orientations of vector bundles and how the formalism of $\A^1$-enumerative geometry allows one to associate to them a well-defined Euler number valued in Grothendieck--Witt of a ground field.

In \autoref{sec:degree}, we compare the Wronski to the section of an appropriate vector bundle over an affine chart on the Grassmannian, and demonstrate that their Brouwer degrees agree up to some global constant. In the case where $m$ and $p$ are both even, we can compute the global $\A^1$-degree of the Wronski using the fact that the Euler classes of relatively oriented vector bundles with odd rank summands are hyperbolic.

Finally, in \autoref{sec:local-index}, we provide an arithmetic formula for the local $\A^1$-degree of the Wronski that holds in all parities. We demonstrate that this local index at an $m$-plane can be interpreted as a ``signed volume'' of the $p$-planes that this $m$-plane intersects. This agrees with and generalizes the local index computed by \cite{SW}. We provide some very preliminary evidence towards a connection between the local $\mathbb{A}^1$-degree of the Wronski and arithmetic Welschinger invariants a la \cite{KLSW}.

\subsection{Acknowledgements}
Thank you to Kirsten Wickelgren for suggesting and supervising this problem, and for Mona Merling for being a constant source of mathematical support. We are immensely grateful to Frank Sottile for inspiring conversations about this work and related topics. Finally, we have benefited from discussions about this work with many people, including Connor Cassady, Andrew Kobin, Stephen McKean, and Sabrina Pauli, to name a few. We acknowledge support from an NSF Graduate Research Fellowship (DGE-1845298).

\section{Preliminaries}\label{sec:preliminaries}

We will begin by delving into the Wronski, understanding its geometric interpretation as counting planes meeting planes of the correct codimension osculating the rational normal curve. By mapping a plane of covectors to the plane it annihilates, we have a natural duality on Grassmannians, and it will benefit us to be able to translate information through this duality, and discuss how it relates to things like Pl\"{u}cker coordinates. After this, we establish some of the foundations of $\A^1$-enumerative geometry, from which we collect the tools to explore the local degree of the Wronski in greater detail.

\subsection{The rational normal curve}\label{subsec:rational-curve}
Over the complex numbers, the degree of the Wronski provides a count of planes which meet a collection of planes, which are said to \textit{osculate the rational normal curve}. We will define these terms, and provide a rough outline of this argument over any field here, but for a more rigorous version of this statement over the complex numbers, we refer the reader to \cite[\sec10.1]{Sottile}.

We may view affine space $\A^{m+p}_k$ as the space $k_{m+p-1}[t]$ of polynomials of degree at most $m+p-1$ with coefficients in $k$ by considering a rational point $(a_0, \ldots, a_{m+p-1}) \in \A^n_k$ as a polynomial $a_0 + a_1 t + \ldots + a_{m+p-1} t^{m+p-1} \in k[t]$. We then let $\gamma: \A^1_k \to \A_k^{m+p}$ denote the \textit{rational normal curve}, also referred to the \textit{moment curve} $\gamma$, defined to be the image of the map
\begin{align*}
    s \mapsto (1, s, s^2, s^3, \ldots, s^{m+p-1}),
\end{align*}
where as above we are identifying affine space with a space of polynomials. That is
\begin{align*}
    \gamma(s) = 1 + st + s^2 t^2 + \ldots + s^{m+p-1}t^{m+p-1} \in k_{m+p-1}[t].
\end{align*}

We may define the derivative of the rational normal curve by deriving termwise, to obtain
\begin{align*}
    \gamma'(s) = (0, 1, 2s, 3s^2, \ldots, (m+p-1)s^{m+p-2}),
\end{align*}
which corresponds to the polynomial
\begin{align*}
    \gamma'(s) = t + 2st^2 + 3s^2 t^3 + \ldots + (m+p-1)s^{m+p-2} t^{m+p-1} \in k_{m+p-1}[t].
\end{align*}
Higher derivatives are defined analogously. One may check that, for any $s$, the elements $\gamma(s)$, $\gamma'(s)$, $\ldots$, $\gamma^{(m+p-1)}(s)$ yield a basis of $k_{m+p-1}[t]$. Thus we obtain an \textit{osculating flag} $F_\bullet(s)$ along the rational curve whose $i$-plane at any time $s$ is the span:
\begin{equation}\label{eqn:osculating-flag-F}
\begin{aligned}
    F_i(s) := \spn\left\{ \gamma(s), \gamma'(s), \ldots, \gamma^{(i-1)}(s) \right\}.
\end{aligned}
\end{equation}
In this setting, we say that the $i$-plane $F_i(s)$ \textit{osculates the rational normal curve} at the point $\gamma(s)$. We will see in \autoref{rmk:flags-in-duality} that $F_m(s)$ is dual in a sense to the planes $E_p(s)$ defined in the introduction.

The monomial basis for polynomials provides an isomorphism between $k_{m+p-1}[t]$ and its dual $k_{m+p-1}[t]^\ast$, given by sending a polynomial $g$ to $g^\ast$, where $g^\ast(f)$ is defined to be the dot product of the coefficients of $g$ and $f$. Under this isomorphism we may view each $\gamma^{(i)}(s)$ as a covector, from which perspective it admits an interesting interpretation.

\begin{proposition}\label{prop:gamma-covector} Considering $\gamma^{(i)}(s)$ as a covector, we see that it has the interpretation of mapping a polynomial to its $i$th derivative evaluated at $s$:
\begin{align*}
    \left(\gamma^{(i)}(s)\right)^\ast (f) = f^{(i)}(s).
\end{align*}
\end{proposition}
\begin{proof} We may compute explicitly for $0 \le j \le m+p-1$ that
\begin{align*}
    \gamma^{(j)}(s) = \sum_{r=j}^{m+p-1} \frac{r!}{(r-j)!} s^{r-j} t^r \in k_{m+p-1}[t].
\end{align*}
Therefore for any $f(t) = \sum_{i=0}^{m+p-1} a_i t^i$, we have that
\begin{align*}
    \left( \gamma^{(j)}(s) \right)^\ast(f) = \sum_{r=j}^{m+p-1} \frac{r!}{(r-j)!} a_{r} s^{r-j} = f^{(j)}(s).
\end{align*}
\end{proof}
We note that we may write the Wronski as a determinant of matrices built out of the rational normal curve and the input polynomials. Let $f_1(t), \ldots, f_m(t) \in k_{m+p-1}(t)$ be $m$ linearly independent polynomials of degree at most $m+p-1$, so that their span defines a point on $\Gr_k(m,m+p)$. Let $f_i(t) = \sum_{j=0}^{m+p-1} a_{i,j} t^j$, and define a matrix $M$ comprised of the coefficients of the polynomials $f_i$:
\begin{align*}
    M &= \left( \begin{array}{c} \text{coefficients of } f_1 \\
       \text{coefficients of } f_2 \\ 
      \vdots \\      \text{coefficients of } f_m \end{array} \right) = \begin{pmatrix} a_{1,0} & a_{1,1} & \cdots & a_{1,m+p-1} \\
       a_{2,0} & a_{2,1} & \cdots & a_{2,m+p-1} \\
      \vdots & \vdots & \ddots & \vdots \\
     a_{m,0} & a_{m,1} & \cdots & a_{m,m+p-1} \end{pmatrix}.
\end{align*}
Let $\Gamma(s)$ denote the matrix $k^m \to k_{m+p-1}[t]$ whose $j$th column is given by the coefficients of the polynomial $\gamma^{(j-1)}(s) \in k_{m+p-1}[t]$
\begin{align*}
    \Gamma(s) = \left(\begin{array}{l | l | l | l} \gamma(s) & \gamma'(s) & \cdots & \gamma^{(m-1)}(s) \end{array}\right).
\end{align*}
Phrased differently, the columns of $\Gamma(s)$ are the basis vectors spanning the $m$-plane $F_m(s)$ osculating the rational normal curve at $\gamma(s)$.

\begin{proposition}\label{prop:Wronski-as-det} In the previous notation, one may express the Wronski polynomial of $f_1, \ldots, f_m$ evaluated at a point $s$ as a determinant:
\begin{align*}
    \det(M\cdot \Gamma(s)) = \Wr(f_1, \ldots, f_m)(s).
\end{align*}
\end{proposition}
\begin{proof} Multiplying a row of $M$ with a column of $\Gamma(s)$ is the same as taking the dot product of $\gamma^{(i)}(s)$ with $f_j(t)$, yielding $f^{(i)}(s)$ by \autoref{prop:gamma-covector}. It follows then that the determinant of the product of $M$ and $\Gamma(s)$ yields the Wronski polynomial evaluated at $s$.
\end{proof}

\begin{corollary}\label{cor:H-meets-Fm} Consider $f_1, \ldots, f_m$ as covectors, let $H$ be the $p$-plane defined by their simultaneous vanishing, and let $s\in k$ be a fixed scalar. Then the Wronski polynomial $\Wr(f_1, \ldots, f_m)(t)$ vanishes at $s$ if and only if $H$ meets $F_m(s)$ non-trivially.
\end{corollary}
\begin{proof} Linear dependence in the columns of $M \cdot \Gamma(s)$ implies that there is a non-trivial linear combination of the covectors $\gamma(s), \ldots, \gamma^{(m-1)}(s)$ which vanishes on each $f_i(t)$. This linear combination provides a point on the intersection of $F_m(s)$, which is the span of these covectors, and on $H$, the plane of covectors annihilating $\spn \left\{ f_1, \ldots, f_m \right\}$.
\end{proof}

We can rephrase this result slightly to state that the plane defined by the span of $f_1, \ldots, f_m$ intersects a $p$-plane dual to $F_m(s)$ nontrivially. In order to make this precise, we must discuss a natural duality arising on Grassmannian varieties.

\subsection{Grassmann Duality}
Let $\Gr_k\left(m,(k_{m+p-1}[t])^\ast \right)$ denote the collection of $m$-planes in the space of linear forms on $k_{m+p-1}[t]$, and let $\spn\left\{ h_1, \ldots, h_m \right\}$ denote a point on this Grassmannian. We may consider the action of the $h_i$'s on $k_{m+p-1}[t]$. The subspace of $k_{m+p-1}[t]$ given by those polynomials $f$ so that $h_1(f) = 0$ is a subspace of codimension 1. The vanishing locus of $m$ linearly independent linear forms imposes $m$ linearly independent conditions, and thus produces a subspace of $k_{m+p-1}[t]$ of codimension $m$. That is to say, each such point $\left\{ h_1, \ldots, h_m \right\}$ canonically defines a $p$-plane in $k_{m+p-1}[t]$. This yields a canonical (i.e. \textit{basis-independent}) isomorphism
\begin{align*}
    \Gr_k \left( m, (k_{m+p-1}[t])^\ast \right) \cong \Gr_k(p, k_{m+p-1}[t]).
\end{align*}

This isomorphism is called \textit{Grassmann duality}. It is an important property of Grassmannians, and for instance can be used to explain the fact that $d(m,p) = d(p,m)$. Grassmann duality is a crucial tool for our geometric interpretation of the Wronski, and shall be used heavily in this paper.

\begin{remark}\label{rmk:flags-in-duality} We may define a flag $E_\bullet(s)$ in $k_{m+p-1}$ where $E_i(s) = \spn\left\{ (t-s)^{m+p-1}, \ldots, (t-s)^{m+p-i-1}\right\}$ is the space of those polynomials which vanish at $s$ to order $\ge m+p-i$. Then for a polynomial $f$, the following are equivalent:
\begin{enumerate}
    \item $f \in E_i(s)$
    \item $(t-s_i)^{m+p-i} \big| f(t)$
    \item $f$, viewed as a linear form, annihilates $F_{m+p-i}(s)$, the osculating plane to the rational normal curve at $s$.
\end{enumerate}
\end{remark}
Thus the flags $E_\bullet(s)$ and $F_\bullet(s)$ are dual \cite[Theorem~10.8]{Sottile}. This allows us to revisit \autoref{cor:H-meets-Fm} to say that $\Wr(f_1, \ldots, f_m)(t)$ vanishes at $s$ if and only if the $m$-plane $W = \spn \left\{ f_1, \ldots, f_m \right\}$ intersects $E_p(s)$ non-trivially. This perspective allows us to develop a geometric intuition for the Wronski.

\begin{proposition} \textit{(Geometric interpretation of the Wronski)} Let $s_1, \ldots, s_{mp}$ be distinct scalars in $k$. Then $\Wr^{-1} \left( \prod_{i=1}^{mp} (t-s_i) \right)$ consists of those $m$-planes which intersect each of $E_p(s_1), \ldots, E_p(s_{mp})$ non-trivially.
\end{proposition}
In particular this tells us that the degree of the Wronski provides a solution to an enumerative problem.

\subsection{Schubert cells and the Pl\"{u}cker embedding} Frequently the Grassmannian can be understood better once it has been embedded in projective space. Viewing the vectors spanning a plane as a wedge power, we can sit the Grassmannian inside a suitably large projective space.

\begin{definition} The \textit{Pl\"{u}cker embedding} for the Grassmannian $\Gr_k(m,k_{m+p-1}[t])$ is defined to be the closed embedding
\begin{align*}
    \Pl : \Gr(m, k_{m+p-1}[t]) &\clhook \Proj \left( \wedge^m k_{m+p-1}[t] \right) \\
    \spn \left\{ f_1, \ldots, f_m \right\} &\mapsto \left[ f_1 \wedge \cdots \wedge f_m \right].
\end{align*}
\end{definition}

\begin{notation} We denote by $\binom{[m+p]}{m}$ the following set of integer sequences:
\begin{align*}
    \binom{[m+p]}{m} &= \left\{ (\alpha_1, \ldots, \alpha_m) \ : \ 1\le \alpha_1 < \alpha_2 < \cdots < \alpha_m \le m+p \right\}.
\end{align*}
\end{notation}

Provided we have \textit{chosen a basis} $e_1, \ldots, e_{m+p}$ for $k_{m+p-1}[t]$, we see that the projective space $\Proj \left( \wedge^m k_{m+p-1}[t] \right) = \P^{\binom{m+p}{m}-1}$ inherits a basis consisting of the coordinates
\begin{align*}
    P_\alpha = e_{\alpha_1} \wedge e_{\alpha_2} \wedge \cdots \wedge e_{\alpha_m},
\end{align*}
where $\alpha$ is varying over all multiindices in $\binom{[m+p]}{m}$. Given any $W = \spn \left\{ f_1, \ldots, f_m \right\}$ on $\Gr_k(m,m+p)$, we may embed it in projective space, where it must be expressible as a $k$-linear sum over the $P_\alpha$'s. We refer to the coefficients appearing in this sum as the \textit{Pl\"{u}cker coordinates} of $W$, and denote them by $z_\alpha(W)$:
\begin{align*}
    f_1 \wedge \cdots \wedge f_m = \sum_{\alpha\in \binom{[m+p]}{m}} z_\alpha(W) P_\alpha.
\end{align*}
How do we compute these $z_\alpha(W)$'s? We remark that we may write the coefficients of $f_1, \ldots, f_m$ in the basis $e_1, \ldots, e_{m+p}$, yielding an $m \times (m+p)$ matrix over $k$. The coefficient $z_\alpha(W)$ associated to a multiindex $\alpha = \left( \alpha_1, \ldots, \alpha_m \right)$ is precisely the determinant of the $m \times m$-minor of this matrix given by the $\alpha_i$th columns. As the image of the Plucker embedding is a projective space, any ambiguities arising in expressing $W$ as a matrix are resolved; that is, the Pl\"{u}cker coordinates corresponding to $W$ are well-defined.

\begin{remark} It is a classical fact that the Pl\"{u}cker embedding is injective; that is, a point on the Grassmannian can be recovered from its Pl\"{u}cker coordinates.
\end{remark}

Grassmann duality translates to duality on Pl\"{u}cker coordinates as well. In order to demonstrate this, we must first define the dual of a multiindex.

\begin{definition} Let $\alpha \in \binom{[m+p]}{m}$ be a multiindex $(\alpha_1, \ldots, \alpha_{m})$. Denote by $\alpha^c \in \binom{[m+p]}{p}$ the complement of $\alpha$ in $(1, \ldots, m+p)$. Then we define the \textit{dual} multiindex $\alpha^\ast$ whose entries are $m+p+1 - (\alpha^c)_i$.
\end{definition}

\begin{example} If $m=2$ and $p=3$, let $\alpha = (1,4)$. Then  $\alpha^c = (2,3,5)$, and $\alpha^\ast = (1,3,4)$.
\end{example}

\begin{proposition}\label{prop:plucker-coords-gr-duality} Consider the Grassmann duality isomorphism
\begin{align*}
    \Gr_k(m, k_{m+p-1}[t]^\ast) \xto{\sim} \Gr_k(p,k_{m+p-1}[t]),
\end{align*}
given by sending an $m$-plane of covectors to the $p$-plane it annihilates. Let $\alpha \in \binom{[m+p]}{m}$, and fix a basis $\left\{ e_i \right\}$ of $k_{m+p-1}[t]$ with dual basis $\left\{ e_i^\ast \right\}$. Then for any $W^\ast \in \Gr_k(m, k_{m+p-1}[t]^\ast)$, where $W$ is the plane it annihilates, we have that
\begin{align*}
    z_{\alpha}(W^\ast) = z_{\alpha^\ast}(W).
\end{align*}
That is, the $\alpha$th Pl\"{u}cker coordinate of $W^\ast$ in the dual basis $e_i^\ast$ is the $\alpha^\ast$th Pl\"{u}cker coordinate of $W$ in the basis $\left\{ e_i \right\}$.
\end{proposition}

\begin{example}\label{ex:plucker-coords-dual-flags} We have that $z_\alpha (F_m(s)) = z_{\alpha^\ast}(E_p(s))$ for any scalar $s$ and any multiindex $\alpha$.
\end{example}

\subsection{Background from \texorpdfstring{$\mathbb{A}^1$}{A1}-enumerative geometry}\label{subsec:a1-degree}

Solving an enumerative problem can often be reduced to the computation of a certain characteristic number of a vector bundle, under certain orientation data and expected dimension assumptions. We first begin with a moduli space of possible solutions to the enumerative problem (for the example of lines on a cubic surface, our moduli space would simply be the Grassmannian of lines in projective 3-space). Following this, we construct an appropriate vector bundle over the moduli space together with a section of the bundle whose zeros are precisely the solutions to the enumerative problem at hand, and which are assumed to be isolated points. In the presence of certain orientation data for the bundle, the solution to our enumerative problem is the Euler class of the bundle, which by the Poincar\'e--Hopf theorem can be thought of as a sum of local indices of the section at points in its zero locus. On a coordinate patch which is compatible with our orientation data, these local indices can be computed as local Brouwer degrees of our section at points in the vanishing locus. Over the complex numbers, the local Brouwer degree at any simple zero will be equal to 1, which we read as a Boolean value informing us that this point on the moduli space is a solution to the enumerative problem (at non-simple points, it will encode the multiplicity of the solution as a natural number). Over other fields, a richer definition of Brouwer degree can produce a wider variety of data at a single solution to an enumerative problem, often revealing deep information about the ambient geometry that was invisible in the complex setting.

The algebrao-geometric analogue of the Brouwer degree that we use is called the $\A^1$\textit{-Brouwer degree}, first defined by Morel \cite{Morel-ICM}, which is valued in the Grothendieck--Witt group of symmetric bilinear forms over $k$. This tool has been instrumental in the development of $\A^1$\textit{-enumerative geometry} (or \textit{enriched enumerative geometry}). This program has grown in recent years due to seminal work of Levine \cite{Levine}, Kass and Wickelgren \cite{KW-EKL}, Bachmann and Wickelgren \cite{BW3}, among others. Recent results include an enriched B\'{e}zout's theorem \cite{Stephen-Bezout}, an enriched count of lines on a quintic threefold \cite{Sabrina}, and a count of conics meeting eight lines \cite{DGGM}. For further reading on this field we refer the reader to the survey papers \cite{Braz-expos,Pauli-expos}.

In order to compute local $\A^1$-degrees of sections of vector bundles, we will first need some analogue of charts from differential topology. This is provided by \textit{Nisnevich coordinates}, defined by \cite[Definition~17]{KW-arithm}.

\begin{definition}\label{def:Nisnevich-coordinates} Let $X$ be a smooth $n$-scheme, $p\in X$ a closed point, and $U \ni p$ an open neighborhood. Then we say that an \'etale map
\begin{align*}
    \phi: U &\to \A^n_k,
\end{align*}
which induces an isomorphism on the residue field at $p$, defines \textit{Nisnevich coordinates near $p$}. We say this defines Nisnevich coordinates \textit{centered at} $p$ if $\phi(p) = 0$.
\end{definition}

\begin{definition}\label{def:Nis-coords-trivialization-tangent-space} Let $X$ be a smooth $n$-scheme admitting Nisnevich coordinates $\phi: U \to \A^n_k = \Spec k[x_1, \ldots, x_n]$ near a point $p \in X$. Affine space admits a standard trivialization, given by the basis elements $\frac{d}{dx_1}, \ldots, \frac{d}{dx_n}$ on $T \A^n_k$. Since $\phi$ is \'etale, it induces an isomorphism
\begin{align*}
    \left. TX \right|_{ U } \xto{\sim} T\A^n_k,
\end{align*}
and by pulling back the basis elements $\frac{d}{dx_i}$, we obtain a basis for $\left. TX \right|_{ U }$. We refer to these basis elements as the \textit{distinguished trivialization of} $\left. TX \right|_{ U }$ arising from the Nisnevich coordinates $\phi$.
\end{definition}

\begin{example}\label{ex:Zar-open-immersion-Nis-coords} If $f:\A^n_k \to X$ is a Zariski open immersion, then by denoting $U:= \im(f)$, the function $U \to \A^n_k$ given by $y \mapsto f^{-1}(y)$ is \'etale, and moreover defines Nisnevich coordinates.
\end{example}

\begin{definition}\label{def:relatively-oriented} (c.f. \cite{OkTel}, \cite{KW-arithm}, \cite[\S4.3]{Morel}) Suppose $E \to X$ is a vector bundle of rank $n$ over a smooth, projective $n$-scheme over a field $k$. Then we say $E$ is \textit{relatively oriented} over $X$ if there is an isomorphism
\begin{align*}
    j: \HOM \left( \det TX, \det E \right) \cong \mathscr{L}^{\otimes 2},
\end{align*}
for $\mathscr{L} \to X$ a line bundle. Any such choice of isomorphism $j$ is called a \textit{relative orientation}.
\end{definition}

\begin{definition} For an open set $U \subseteq X$, and a relatively oriented bundle $(E,j)$ we say a section $s\in \Gamma \left( U, \HOM \left( \det TX, \det E \right) \right)$ is \textit{a square} if its image in $\Gamma(U, \mathscr{L}^{\otimes 2})$ is a square, meaning it is of the form $s' \otimes s'$ for some $s' \in \Gamma(U, \mathscr{L})$.
\end{definition}

Now suppose we had Nisnevich coordinates $\phi: U \to \A^n_k$ near a point $p\in X$, and a relative orientation $j: \HOM \left( \det TX, \det E \right) \xto{\sim} \mathscr{L}^{\otimes 2}$.
As in \autoref{def:Nis-coords-trivialization-tangent-space}, the coordinates $\phi$ induce a trivialization of $\left. TX \right|_{ U }$, and by restricting $U$ we may assume that there is a trivialization of the vector bundle $E$ over $U$, meaning an isomorphism $\psi: \left. E \right|_{ U } \cong \A^n_k$.

\begin{definition} In the situation above, we say the trivialization $\psi$ is \textit{compatible} with the Nisnevich coordinates $\phi$ \textit{and} the relative orientation $(E,j)$ if the associated element in
\begin{align*}
    \HOM \left( \det \left. TX \right|_{ U }, \det \left. E \right|_{ U }  \right)
\end{align*}
taking our distinguished basis of $\det \left. TX \right|_{ U }$ to the distinguished basis of $\det \left. E \right|_{ U }$ is a square.
\end{definition}

If $\sigma: X \to E$ is a section, $p$ is an isolated zero of $\sigma$, and $U \ni p$ an open neighborhood not containing any other points of $Z(\sigma)$, we can pull back the map $\psi \circ \sigma_U$ by $\phi$ to an endomorphism of affine space, which we denote by $(f_1, \ldots, f_n)$, yielding the following diagram:
\[ \begin{tikzcd}
    U\dar["\phi" left]\rar["\left. \sigma \right|_{ U }" above] & \left. E \right|_{ U }\dar["\psi" right]\\
    \A^n_k\rar[dashed,"{(f_1, \ldots, f_n)}" below] & \A^n_k.
\end{tikzcd} \]

\begin{definition}\label{def:local-index-sigma} The \textit{local index} of $\sigma$ at $p$ is defined by
\begin{align*}
    \ind_p \sigma = \deg_{\phi(p)}^{\A^1} (f_1, \ldots, f_n),
\end{align*}
where $\deg_{\phi(p)}^{\A^1}(f)$ is the local $\A^1$-Brouwer degree of $f$ at $\phi(p)$, that is, it is a class in the Grothendieck--Witt group $\GW(k)$. 
\end{definition}
For techniques and code for computing such local degrees, we refer the reader to \cite{Bezoutians}.

\begin{definition}\cite[Definition 33]{KW-arithm} Let $E \to X$ be a relatively oriented vector bundle of rank $r$ over a smooth $r$-dimensional scheme $X \in \Sch_k$, and let $\sigma : X \to E$ be a section of the bundle with isolated zeros so that Nisnevich coordinates exist near every zero. We define the \textit{Euler number}
\begin{align*}
	e(E,\sigma) = \sum_{p\in Z(\sigma)} \ind_p \sigma,
\end{align*}
where we are summing over closed points $p$ where $\sigma$ vanishes.
\end{definition}

\begin{proposition} (\cite[Theorem~1.1]{BW3}) The Euler class of a relatively oriented vector bundle $e(E,\sigma)$ over a smooth and proper scheme $X$ is independent of the choice of section.
\end{proposition}

We can also ask about how we might wield Nisnevich coordinates to compute a global $\A^1$-degree of a morphism between suitably nice smooth $n$-schemes. This notion is based off forthcoming work of Kass, Levine, Solomon and Wickelgren \cite{KLSW}, and was discussed briefly in the expository paper \cite[\S8]{Pauli-expos}. See also \cite[2.53]{Morel-Sawant} for the degree of an endomorphism of projective space, and \cite{Morel} for the degree of an endomorphism of $\P^n/\P^{n-1}$.

\begin{definition}\label{def:relative-orientation-map} (\cite{KLSW}) Let $f: X \to Y$ be a finite map of smooth $n$-schemes over a field $k$. We say that $f$ is \textit{oriented} if $\HOM \left( \det TX, \det f^\ast TY \right)$ is a relatively oriented vector bundle over $X$. Phrased differently, a \textit{relative orientation} for $f$ is a choice of isomorphism
\begin{align*}
    j: \HOM(\det TX, \det f^\ast TY) \xto{\sim} \mathscr{L}^{\otimes 2},
\end{align*}
for $\mathscr{L} \to X$ a line bundle over $X$.
\end{definition}

\begin{definition} (\cite{KLSW})  Let $U \subseteq X$ and $V \subseteq Y$ be open sets such that $f(U) \subseteq V$. We say that Nisnevich coordinates $\phi: U \to \A^n_k$ and $\psi:V \to \A^n_k$ are \textit{compatible with the relative orientation} $j$ if the distinguished section of
\begin{align*}
    \Gamma(U, \HOM(\det \left. TX \right|_{ U }, \det \left. f^\ast TY \right|_{ U })
\end{align*}
is a square.
\end{definition}

\begin{theorem} (\cite{KLSW}, c.f. \cite[8.7]{Pauli-expos}) For a finite oriented map $f: X \to Y$ of smooth $k$-schemes, with $Y$ an $\A^1$-connected scheme, we have a well-defined degree valued in $\GW(k)$, defined by $\deg^{\A^1}(f) = \sum_{p\in f^{-1}(q)} \deg_p^{\A^1}(f)$.
\end{theorem}

\section{The $\A^1$-degree of the Wronski}\label{sec:degree}

Our strategy for computing the $\A^1$-degree of the Wronski is to exhibit a section $\sigma$ of a particular vector bundle $\V \to \Gr_k(m,m+p)$, and on a suitable open chart, equate $\sigma$ with the Wronski up to another morphism of constant $\A^1$-degree. In this way, we will be able to equate the local index of the section $\sigma$ with a constant multiple of the local $\A^1$-degree of the Wronski in all possible parities (\autoref{lem:local-index-is-a1-local-degree-wronski}). In the case when $m$ and $p$ are even, the bundle $\mathcal{V}$ will be relatively orientable, and its Euler class will therefore be an integer multiple of the hyperbolic form in $\GW(k)$ (\autoref{lem:Levine-euler-class}), as will the global $\A^1$-degree of the Wronski. Deferring to the classical computation of Schubert, as we know the rank of this form over $\mathbb{C}$, we are able to provide the global $\A^1$-degree of the Wronski in \autoref{thm:deg-wronski-both-even}.

\subsection{Nisnevich coordinates on the Grassmannian, distinguished bases}

We will begin by establishing the existence of Nisnevich coordinates on an arbitrary Grassmannian. Let $W \in \Gr_k(m,m+p)$ be an arbitrary point, and pick a basis $e_1, \ldots, e_{m+p}$ of $k_{m+p-1}[t]$ so that
\begin{align*}
    W = \spn \left\{ e_{p+1}, \ldots, e_{m+p} \right\}.
\end{align*}

\begin{definition} We define a \textit{moving basis around} $W$, denoted by $\left\{ \widetilde{e}_1, \ldots, \widetilde{e}_{m+p} \right\}$, to be a basis of $k_{m+p-1}[t]$, parametrized by $\A^{mp}_k = \Spec[ x_{i,j}]_{1\le i\le m,\ 1\le j\le p}$:
\begin{equation}\label{eqn:moving-basis}
\begin{aligned}
    \til{e}_i &=  \begin{cases} e_i & 1\le i\le p \\ e_i + \sum_{j=1}^p x_{i-p,j}e_j & p+1 \le i\le m+p. \end{cases}
\end{aligned}
\end{equation}
\end{definition}

Consider the morphism
\begin{align*}
    \A^{mp}_k = \Spec[ x_{i,j}]_{1\le i\le m,\ 1\le j\le p} &\to \Gr_k(m,m+p) \\
    (x_{i,j})_{i,j} &\mapsto \spn \left\{ \widetilde{e}_{p+1}, \ldots, \widetilde{e}_{m+p} \right\}.
\end{align*}

Another way to phrase the image of this map is that $(x_{i,j})$ is sent to
\begin{equation}\label{eqn:local-coords}
\begin{aligned}
    \text{RowSpace}  \begin{pmatrix} x_{1,1} & \cdots & x_{1,p} & 1 & 0 & \cdots & 0 & \\
    x_{2,1} & \cdots & x_{2,p} & 0 & 1 & \cdots & 0 \\
    \vdots & \ddots & \vdots  & \vdots &\vdots &\ddots &\vdots \\
    x_{m,1} & \cdots & x_{m,p} & 0 & 0 & \cdots & 1
    \end{pmatrix}.
\end{aligned}
\end{equation}
where the columns correspond to the basis elements $e_1, \ldots, e_{m+p}$.
This define a Zariski open immersion by the content of the argument in \cite[Lemma~40]{KW-arithm}. Letting $U$ denote its image, we obtain Nisnevich coordinates centered around $W$ by \autoref{ex:Zar-open-immersion-Nis-coords}.

\begin{remark}\label{rmk:labelname} We can provide a more classical description of the Nisnevich coordinates above. We note that the $m$-plane $W$ lives inside the ambient vector space, so we have a short exact sequence
\begin{align*}
    W \hookto k_{m+p-1}[t] \to k_{m+p-1}[t]/W.
\end{align*}
Picking a splitting for this is equivalent to picking a complementary $p$-plane to $W$. We remark that, due to the construction of the tangent space to the Grassmannian at $W$, we have an isomorphism
\begin{align*}
    T_W \Gr_k(m,k_{m+p-1}[t]) \cong \Hom(W,k_{m+p-1}[t]/W).
\end{align*}
By fixing such a complementary plane (i.e. choosing a basis $e_1, \ldots, e_p$), we can identify $\Hom(W,k_{m+p-1}[t]/W)$ with $\A^{mp}_k$, by sending a homomorphism to its graph. This graph is precisely given by the matrix in \autoref{eqn:local-coords}. In this sense, our open cell $U$ is the subspace of $m$-planes in the Grassmannian which only intersect the $p$-plane $\spn \left\{ e_1, \ldots, e_p \right\}$ trivially (c.f. \cite[\S3.2.2]{3264}).
\end{remark}

\begin{remark}\label{rmk:xij-from-pluck-coords} We remark that the $x_{i,j}$'s can be recovered as particular \textit{Pl\"ucker coordinates}. Let $W$ denote the $m$-plane corresponding to $(x_{i,j})$. Consider the $k \times k$ minor of columns $(j, p+1, \ldots,\widehat{p+i}, \ldots, p+m)$. Expanding along the first column, we see that everything vanishes until we hit the $i$th row, at which point the determinant yields $(-1)^i x_{i,j}$. That is,
\begin{align*}
    x_{i,j} &= (-1)^i p_{(j,p+1, \widehat{p+i}, \ldots, p+m)}(W). 
\end{align*}
Here the Pl\"{u}cker coordinates are taken with respect to the basis $\left\{ e_1, \ldots, e_{m+p} \right\}$. For concision we will introduce new notation to correspond to this multiindex:
\begin{equation}\label{eqn:alpha-multiindex}
\begin{aligned}
    \alpha(i,j) := (j,p+1, \ldots, \widehat{p+i}, \ldots, p+m).
\end{aligned}
\end{equation}
\end{remark}

Nisnevich coordinates defined by a moving basis induce distinguished basis elements on the tangent space $\left. T\Gr_k(m,m+p) \right|_{ U }$. In order to describe these distinguished basis elements, we first must discuss the structure of the tangent space of the Grassmannian.

\begin{definition} The \textit{tautological bundle} on the Grassmannian, denoted $\mathcal{S} \to \Gr_k(m,m+p)$, is the $m$-plane bundle whose fiber over the point $W$ is the vector space $W$ itself.
\end{definition}
The line bundle $\O(1)$ on the Grassmannian, defined to be the pullback of $\O(1)$ under the Pl\"{u}cker embedding, is precisely $\det \mathcal{S} = \wedge^m \S$. Including the tautological bundle into the trivial rank $mp$ bundle, we obtain the \textit{quotient bundle}, defined as the cokernel
\begin{align*}
    0 \to \mathcal{S} \to \A^{mp}_k \tto \mathcal{Q} \to 0.
\end{align*}
We can therefore express the tangent bundle of the Grassmannian by
\begin{equation}\label{eqn:tangent-space-Gr-hom-bundle}
\begin{aligned}
    T\Gr_k(m,m+p) \cong \HOM(\mathcal{S},\mathcal{Q}) = \mathcal{S}^\ast \otimes \mathcal{Q}.
\end{aligned}
\end{equation}

\begin{proposition}\label{prop:distinguished-bases-Gr} Given Nisnevich coordinates $U \to \A^{mp}_k$ corresponding to a moving basis $\widetilde{e}_1, \ldots, \widetilde{e}_{m+p}$, then one has the following distinguished bases over $U$:
\begin{enumerate}
    \item $\left\{ \widetilde{e}_{p+1}, \ldots, \widetilde{e}_{p+m} \right\}$ is a distinguished basis for the tautological bundle $\left.\mathcal{S}\right|_U$
    \item Letting $\widetilde{\phi}_i$ denote the cobasis element to $\widetilde{e}_i$, we see that $\left\{ \widetilde{\phi}_{p+1}, \ldots, \widetilde{\phi}_{m+p}  \right\}$ provides a distinguished basis for the dual tautological bundle $\left. \mathcal{S}^\ast \right|_{ U }$.
    \item $\left\{\widetilde{e}_1, \ldots, \widetilde{e}_p\right\}$ provides a distinguished basis for the quotient bundle $\left. \mathcal{Q} \right|_{ U }$.
    \item The tensor products of vectors
    \begin{align*}
        \left\{ \widetilde{\phi}_j \otimes \widetilde{e}_i \ : \ 1\le i\le p \text{ and } p+1\le j\le m+p \right\}
    \end{align*}
    provide a distinguished basis for the tangent bundle $\left.T\Gr_k(m,m+p)\right|_U$.
\end{enumerate}
\end{proposition}

\begin{lemma}\label{lem:TGr-orientable-even-even} If $m$ and $p$ are both even, there is a global orientation of the tangent bundle $T\Gr_k(m,m+p)$ which is compatible with any Nisnevich coordinates defined by moving bases.
\end{lemma}
\begin{proof} This is a direct generalization of \cite[Lemma~8]{SW}. Let $(e_1,\ldots,e_{m+p})$ and $(e_1',\ldots,e_{m+p}')$ denote two bases of $k_{m+p-1}[t]$, and let $\{\til{e}_i\}$, $\{\til{e}_i'\}$ denote the associated moving bases parametrizing open cells $U$ and $U'$ of the Grassmannian. If $U\cap U'\neq \emptyset$, we have that
\begin{align}\label{eqn:overlap-equal-basis-spn}
	\spn\{\til{e}_{p+1},\ldots,\til{e}_{m+p}\} = \spn\{\til{e}_{p+1}',\ldots,\til{e}_{m+p}'\} \quad \quad \text{ on } U\cap U' .
\end{align}

Letting $\til{\phi}_i$ and $\til{\phi}_i'$ denote the dual basis elements, respectively, we obtain canonical trivializations for $\left. T\Gr_k(m,m+p) \right|_{U\cap U'}$, given by:
\begin{align*}
	&\{\til{\phi}_j \otimes \til{e}_i \ : \ 1\leq i \leq p,\ p+1\leq j\leq m+p\}, \\
	&\{\til{\phi}_j' \otimes \til{e}_i' \ : \ 1\leq i \leq p,\ p+1\leq j\leq m+p\}.
\end{align*}
Denote by $B$ the change of basis matrix from $\left\{ \widetilde{e}_1, \ldots, \widetilde{e}_p \right\}$ and $\left\{ \widetilde{e}_1', \ldots, \widetilde{e}_p' \right\}$ on the quotient bundle $\left.\mathcal{Q}\right|_{U\cap U'}$, and denote by $A$ the  change of basis matrix on $\left.\mathcal{S}^\ast \right|_{U\cap U'}$ from the basis $\left\{ \widetilde{\phi}_{p+1}, \ldots, \widetilde{\phi}_{m+p} \right\}$ to $\left\{ \widetilde{\phi}_{p+1}', \ldots, \widetilde{\phi}_{m+p}' \right\}$. Then the change of basis matrix on $\left. T\Gr_k(m,m+p) \right|_{U\cap U'}$ is given by $A \otimes B$. The determinant of this matrix is $\det(A)^m \det(B)^p$. As $m$ and $p$ are both even, this is a square in $\O(U\cap U')^\times$.
\end{proof}

\subsection{Relative orientations of bundles over the Grassmannian; even-even parity}

We will discuss a bundle $\mathcal{V}$ over the Grassmannian, which is relatively orientable in the case where $m$ and $p$ are both even. We will additionally construct a section $\sigma : \Gr_k(m,m+p) \to \mathcal{V}$, whose local degree at any simple zero is related to the local $\A^1$-degree of the Wronski in any parities.

\begin{notation} We denote by $\mathcal{V}$ the $mp$-dimensional line bundle
\begin{equation}\label{eqn:definition-bundle-V}
\begin{aligned}
     \mathcal{V} = \bigoplus_{i=1}^{mp} \bigwedge^m \mathcal{S}^\ast \to \Gr_k(m,m+p).
\end{aligned}
\end{equation}
\end{notation}

\begin{proposition}\label{prop:V-orientable} The vector bundle $\V \to \Gr_k(m,m+p)$ is relatively orientable if and only if $m$ and $p$ are both even.
\end{proposition}
\begin{proof} For our bundle $\V$, we compute that
\begin{align*}
	\HOM(\det T\Gr(m,m+p), \det \V) &\cong \HOM(\O(m+p), \prod_{i=1}^{mp} \det(\mathcal{E})) \cong \HOM(\O(m+p),\O(mp)) \\
	&\cong \O(-m-p)\otimes \O(mp) \cong \O(mp - m - p).
\end{align*}
We note that $\O(mp - m - p)$ is a square of a line bundle if and only if $mp - m - p \equiv 0\pmod{2}$, that is, $m$ and $p$ are both even.
\end{proof}

\begin{proposition} If $m$ and $p$ are both even, then there is a relative orientation of the vector bundle $\mathcal{V} \to \Gr_k(m,m+p)$ which is compatible with Nisnevich coordinates defined by moving bases.
\end{proposition}
\begin{proof} Take two cells $U$ and $U'$ on $\Gr_k(m,m+p)$ with non-empty intersection, parametrized respectively by the moving bases $\widetilde{e}_i$ and $\widetilde{e}_i'$, and assume as before that
\begin{align*}
    \spn \left\{ \widetilde{e}_{p+1}, \ldots, \widetilde{e}_{m+p} \right\} = \spn \left\{ \widetilde{e}_{p+1}', \ldots, \widetilde{e}_{m+p}' \right\} \text{ on } U \cap U'.
\end{align*}
The trivializations $\left\{ \widetilde{\phi}_{p+1}, \ldots, \widetilde{\phi}_{m+p} \right\}$ and $\left\{ \widetilde{\phi}_{p+1}', \ldots, \widetilde{\phi}_{m+p}' \right\}$ on the dual tautological bundle $\left. \mathcal{S}^\ast \right|_{ U\cap U' }$ induce associated trivializations $\widetilde{\phi}_{p+1}\wedge \cdots \wedge \widetilde{\phi}_{m+p}$ and $\widetilde{\phi}_{p+1}'\wedge \cdots \wedge \widetilde{\phi}_{m+p}'$, respectively, for $\left.\wedge^m \mathcal{S}^\ast \right|_{U\cap U'}$. If $A$ denotes the change of basis matrix on $\left. \mathcal{S}^\ast \right|_{ U \cap U' }$ as above, then $\det(A)$ denotes the change of basis on $\left.\wedge^m \mathcal{S}^\ast \right|_{U\cap U'}$. Since $\mathcal{V} = \oplus_{i=1}^{mp} \wedge^m \mathcal{S}^\ast$, we have that the change of basis on $\left. \mathcal{V} \right|_{ U\cap U' }$ is given by a block sum of $mp$ copies of $\det(A)$. Thus the change of basis matrix in $\HOM( \det T \Gr_k(m,m+p), \det \mathcal{V}) \cong (\det T\Gr_k(m,m+p))^\ast \otimes \det \mathcal{V}$ over $U \cap U'$ is given by 
\begin{align*}
    \det \left( A \otimes B \right)^{-1} \otimes \det \left( \bigoplus_{i=1}^{mp} \det(A)\right) &= \det(A)^{-m} \det(B)^{-p} \det(A)^{mp}.
\end{align*}
As $m$ and $p$ are both even, this is a square.
\end{proof}

\subsection{Interpreting the Wronski as a section of a line bundle}\label{sec:wronski-as-section} We now construct a section $\sigma$ of the bundle $\mathcal{V}$ which is intimately related to the Wronski. For this section we fix $s_1, \ldots, s_{mp}$ to be distinct scalars in $k$ --- the reader is invited to think of these scalars as timestamps on $\A^1_k$, yielding positions on the rational normal curve at each time, as well as osculating planes.

Recall from \autoref{prop:gamma-covector} the covector $\left( \gamma^{(j)}(s) \right)^\ast$, which mapped a polynomial $f$ to $f^{(j)}(s)$. We would like to consider these covectors as $j$ ranges from $0$ to $m-1$, and as $s$ varies over our set of scalars $\left\{ s_1, \ldots, s_{mp} \right\}$. To that end, it will be beneficial to introduce some more compact notation.

\begin{notation} We denote by $\sigma_{i,j}$ the covector $\left( \gamma^{(j-1)}(s_i) \right)^\ast$, given by
\begin{align*}
    \sigma_{i,j} : k_{m+p-1}[t] &\to k \\
    f &\mapsto f^{(j-1)}(s_i).
\end{align*}
Note the indexing on $\sigma_{i,j}$: the index $i$ is running from $1$ to $mp$, keeping track of the time on the rational normal curve, while $j$ is running from $1$ to $m$, indicating the extent to which the input is being differentiated.
\end{notation}

\begin{remark} For a fixed $i$, the covectors $\left\{ \sigma_{i,1}, \sigma_{i,2}, \ldots, \sigma_{i,m} \right\}$ cut out a $p$-plane under Grassmannian duality. This plane is precisely $E_p(s_i)$, as defined in \autoref{rmk:flags-in-duality}.
\end{remark}

\begin{notation} We denote by $\sigma_i$ the wedge of covectors $\sigma_{i,1} \wedge \cdots \wedge \sigma_{i,m}$. This is a section of $\O(1)$ over the Grassmannian, that is, $\sigma_i : \Gr_k(m,m+p) \to \wedge^m \mathcal{S}^\ast$. Letting $i$ vary from $1$ to $mp$, we obtain $mp$ sections of $\wedge^m \mathcal{S}^\ast$, that is, a section of our bundle $\mathcal{V}$. We denote by $\sigma$ this section:
\begin{align*}
    \sigma := \bigoplus_{i=1}^{mp} \sigma_i = \bigoplus_{i=1}^{mp} \left( \wedge_{j=1}^m \sigma_{i,j} \right) : \Gr_k(m,m+p) \to \mathcal{V} = \bigoplus_{i=1}^{mp} \wedge^m \mathcal{S}^\ast.
\end{align*}
We may also write $\sigma = \sigma(s_1, \ldots, s_{mp})$ if we wish to indicate dependence of $\sigma$ on the initial choice of scalars $s_i$.
\end{notation}

\begin{proposition}\label{prop:zeros-sigma-Wr} We see that $W = \spn\left\{ f_1, \ldots, f_m \right\}$ is a zero of $\sigma_i$ if and only if the Wronski polynomial $\Wr(f_1, \ldots, f_m)(t)$ vanishes at $s_i$.
\end{proposition}
\begin{proof} We observe that $\sigma_i(W)$ vanishes if and only if $(\sigma_{i,1}\wedge \cdots \wedge \sigma_{i,m})(f_1 \wedge \cdots \wedge f_m) = 0$. This evaluation of wedges of covectors can be computed as
\begin{align*}
    \sigma_i(W) &= (\sigma_{i,1}\wedge \cdots \wedge \sigma_{i,m})(f_1 \wedge \cdots \wedge f_m) \\
    &= \begin{vmatrix} f_1(s_i) & f_2(s_i) & \cdots & f_m(s_i) \\
   f_1'(s_i) & f_2'(s_i) & \cdots & f_m'(s_i) \\
  \vdots & \vdots & \ddots & \vdots \\
 f_1^{(m-1)}(s_i) & f_2^{(m-1)}(s_i) & \cdots & f_m^{(m-1)}(s_i) \end{vmatrix} \\
 &= \Wr(f_1, \ldots, f_m)(s_i).
\end{align*}
\end{proof}

\begin{corollary}\label{cor:wronski-intersection-interpretation} Consider the class of the polynomial $\Phi(t) = \prod_{i=1}^{mp}(t-s_i)$ in projective space $\P^{mp}_k$. We have that the following are equivalent for a point $W = \spn\left\{ f_1, \ldots, f_m \right\} \in \Gr_k(m,m+p)$:
\begin{enumerate}
    \item $W$ has nonempty intersection with each of $E_p(s_1), \ldots, E_p(s_{mp})$.
    \item $W$ is a zero of the section $\sigma : \Gr_k(m,m+p) \to \V$.
    \item $\Wr(f_1, \ldots, f_m)(t)$, as a polynomial in $t$, has a root at each $s_i$ for $1\le i\le mp$.
    \item $W$ lives in the fiber $\Wr^{-1} \left( \Phi(t) \right)$.
\end{enumerate}
\end{corollary}

\subsection{Big open cells}\label{sec:cells}

Let $Y \subseteq \P_k^{mp}$ denote the collection of monic polynomials in $\Proj\ k_{mp}[t]$ of degree equal to $mp$. This defines an open affine cell of projective space, of dimension $mp$. Denote by $X = \Wr^{-1}(Y)$ the preimage of this cell in the Grassmannian.

\begin{remark}\label{rmk:big-open-cell-properties} We refer to $X$ as the \textit{big open cell}, and remark a few properties about it.
\begin{enumerate}
    \item This is is a coordinate patch parametrized around the point $\spn \left\{ t^p, \ldots, t^{m+p-1} \right\}$, and therefore $X \cong \A^{mp}_k$.
    \item This is the \textit{big open cell} as defined in \cite[p.5]{EG}, from where we took the terminology.
    \item A point $W\in \Gr_k(m,m+p)$ lies in the open cell $X$ if and only its Wronski polynomial is of degree $mp$.
\end{enumerate}
\end{remark}

By this very last point, if $\Phi(t) := \prod_{i=1}^{mp}(t-s_i)$, then in order to study the fiber $\Wr^{-1}(\Phi(t))$, it suffices to restrict our attention to the big open cell $X$. Let $W \in X$ be an arbitrary point, and fix $e_{p+1}, \ldots, e_{p+m}$ to be monic polynomials which span $W$. Extending this to a basis $e_1, \ldots, e_{m+p}$ of $k_{m+p-1}[t]$, we can parametrize an open cell $U\cong \A^{mp}_k$ centered around $W$. For degree reasons, we observe that $U \subseteq X$, so that we have an induced map $\left. \Wr \right|_{ U } : U \to Y$.
\begin{remark} Let $W \in X$, and let $U$ be an affine cell parametrized around $W$ by a moving basis. Then the restricted Wronski $\left. \Wr \right|_{ U }: U \to Y$ has an orientation which is induced by the trivializations of $TU$ and $TY$.
\end{remark}
Thus we see that $\left. \Wr \right|_{ U }$ is a map of the form $\A^{mp}_k \to \A^{mp}_k$. What does this map look like? If $\left( x_{ij} \right)$ is a point on $\A^{mp}_k \cong U$, we have that its Wronski polynomial is a degree $mp$ polynomial of the form
\begin{align*}
    \Wr \left( \til{e}_{p+1}(x), \ldots, \til{e}_{p+m}(x) \right)(t) = \sum_{i=0}^{mp} h_i t^i.
\end{align*}
This is by definition a point $\left[ h_0: \ldots : h_{mp} \right]$ in projective space. In order to take the affine chart $Y$ we must divide out by $h_{mp}$, which we know to be non-zero by \autoref{rmk:big-open-cell-properties} since $W$ lies on $X$. Moreover since we picked the $e_{p+i}$'s to be monic, we know exactly what $h_{mp}$ is! By only picking out the highest degree terms in the Wronski, we can observe that $h_{mp}$ is the coefficient on the monomial $\Wr \left( t^{p}, t^{p+1}, \ldots, t^{m+p-1} \right)$, which is well-defined over $k$ via our hypothesis that $(m+p-1)!$ is invertible over $k$. It is well-known that this is the Vandermonde $V(p,p+1, \ldots, m+p-1)$ \cite[Lemma~1]{Bostan-Dumas}, and a simple induction argument shows that this is equal to $\prod_{i=1}^{m-1} i!$. We will now compare the local section $\sigma$ to the Wronski. In order to do this, we must first introduce some notation.

\begin{notation} We define the following maps from $\A^{mp}_k$ to itself:
\begin{itemize}
    \item By abuse of notation, denote by $V_{m,p}: \A^{mp}_k \to \A^{mp}_k$ the map which multiplies each coordinate by the scalar $\prod_{i=1}^{m-1} i!$.
    \item Denote by $\ev_s: \A^{mp}_k \to \A^{mp}_k$ the map which sends a tuple $(a_0, \ldots, a_{m+p-1})$, viewed as a polynomial $g(t) = \sum_{i=0}^{mp-1} a_i t^i$ to the tuple $\left( g(s_1), \ldots, g(s_{mp}) \right)$.
    \item Finally, denote by $\tr_s$ the translation map
    \begin{align*}
        \tr_s: \A^{mp}_k &\to \A^{mp}_k \\
        (x_1, \ldots, x_{mp}) &\mapsto \left( x_1 + V_{m,p} s_1^{mp}, \ldots, x_{mp} + V_{m,p} s_{mp}^{mp} \right).
    \end{align*}
\end{itemize}

\end{notation}

\begin{lemma}\label{lem:commutative-diagram-section-wronski} Let $W\in X$, and let $U$ be an open cell parametrized around $X$, determined by a monomial basis as above. Then the following diagram commutes:
\[\begin{tikzcd}
     &  & {\left.\mathcal{V}\right|_U}\ar[drr,"{\oplus_{i=1}^{mp}(\til{e}_{p+1}\wedge \cdots \wedge \til{e}_{m+p})}" above right] &  & \\
    U\ar[urr,"{\left. \sigma \right|_{ U }}" above left]\ar[dr,"{\left. \Wr \right|_{ U }}" below left] &  &  &  & \A^{mp}_k\\
     & Y\rar["{V_{m,p}}" below] & \A^{mp}\rar["{\ev_s}" below] & \A^{mp}\ar[ur,"{\text{tr}_s}" below right] &
\end{tikzcd} \]
\end{lemma}
\begin{proof} Fix $e_1, \ldots, e_{mp}$ as desired, and begin with a point $(x_{ij})$ on the affine space $U$. Let its Wronski be written as
\begin{align*}
    \Wr \left( \til{e}_{p+1}(x), \ldots, \til{e}_{p+m}(x) \right) = \sum_{i=0}^{mp} h_i t^i,
\end{align*}
where we have that $h_{mp} = \prod_{i=1}^{m-1}i!$ as above. Landing in $Y$, we have that $(x_{i,j})$ is mapped to the $mp$-tuple
\begin{align*}
    \left( \frac{h_0}{h_{mp}}, \ldots, \frac{h_{mp-1}}{h_{mp}} \right).
\end{align*}
Applying the map $V_{m,p}$, we multiply each factor through by $h_{mp}$ (which we remark is a constant which is independent of $(x_{ij})$), which clears denominators and maps us to
\begin{align*}
    \left( h_0, h_1, \ldots, h_{mp-1} \right).
\end{align*}
Applying the evaluation map $\ev_s$, we arrive at
\begin{align*}
    \left( \sum_{i=0}^{mp-1} h_i s_1^i, \sum_{i=0}^{mp-1} h_i s_2^i, \ldots, \sum_{i=0}^{mp-1} h_i s_{mp}^i \right).
\end{align*}
Finally applying our translation map, we obtain
\begin{align*}
    \left( \sum_{i=0}^{mp} h_i s_1^i, \ldots, \sum_{i=0}^{mp} h_i s_{mp}^i \right) = \left( \Wr\left(\til{e}_{p+1}(x), \ldots, \til{e}_{p+m}(x)\right)(s_1), \ldots, \Wr\left(\til{e}_{p+1}(x), \ldots, \til{e}_{p+m}(x)\right)(s_{mp})  \right).
\end{align*}
However we remark that by \autoref{prop:zeros-sigma-Wr} this is exactly what we obtain by applying $\sigma$ to the point $(x_{ij})$ and trivializing $\mathcal{V}$ over $U$.
\end{proof}

\begin{remark} The global $\A^1$-degree of the map $V_{m,p}$ is $\left\langle \left( \prod_{i=1}^{m-1} i! \right)^{mp} \right\rangle$, since we are simply multiplying the scalar $\prod_{i=1}^{m-1} i!$ into each of the $mp$ coordinates. The global degree of the translation map $\tr_s$ is just $\left\langle 1 \right\rangle$, since translation is $\A^1$-homotopic to the identity.
\end{remark}

\begin{lemma}\label{lem:vandermonde-evaluation} The global $\A^1$-degree of the evaluation map $\ev_{s_1, \ldots, s_{mp}}$ is precisely
\begin{align*}
    \deg^{\A^1} \ev_{(s_1, \ldots, s_{mp})} = \left\langle V(s) \right\rangle,
\end{align*}
where $V(s):=V(s_1, \ldots, s_{mp})$ denotes the Vandermonde determinant. As a result, since this is a rank one element of $\GW(k)$, this is the local $\A^1$-degree at any root of $\ev_{(s_1, \ldots, s_{mp})}$.
\end{lemma}
\begin{proof} Let $\ev_{(s_1, \ldots, s_{mp})} = (\ev_1, \ldots, \ev_{mp})$, and let $(a_0, \ldots, a_{mp-1})$ correspond to $a_0 + a_1 t + \ldots + a_{mp-1}t^{mp-1} + t^{mp}$. Then we can see $\frac{\partial \ev_j}{\partial a_i} = s_j^{i-1}$.
\end{proof}

\begin{lemma}\label{lem:local-index-is-a1-local-degree-wronski} Let $s_1, \ldots, s_{mp}$ be distinct, and let $\Phi(t) = \prod_{i=1}^{mp}(t-s_i)$. For any $[W] \in \Wr^{-1}(\Phi(t))$, we have that
\begin{align*}
    \ind_W \sigma = \left\langle V(s) \cdot  \left( \prod_{i=1}^{m-1} i! \right)^{mp}\right\rangle \cdot \deg^{\A^1}_W \Wr.
\end{align*}
\end{lemma}
\begin{proof} The proof follows from applying the local degree to the commutative diagram in \autoref{lem:commutative-diagram-section-wronski}, and deferring to the computation in \autoref{lem:vandermonde-evaluation}.
\end{proof}

An explicit formula for $\ind_W \sigma$ at any simple root will be provided in Section \ref{sec:local-index}.

\begin{lemma}\label{lem:Levine-euler-class} If $\V$ is relatively orientable, then its Euler class $e(\V) \in \GW(k)$ is an integer multiple of the hyperbolic element $\mathbb{H}$.
\end{lemma}
For a proof of this lemma, see \cite[4.3]{Levine} as well as the discussion in \cite[Section~4]{SW}.

\begin{theorem}\label{thm:deg-wronski-both-even} Let $m$ and $p$ be even. Then the $\A^1$-degree of the Wronski is
\begin{align*}
    \deg^{\A^1} \Wr &= \frac{d(m,p)}{2}\mathbb{H}.
\end{align*}
\end{theorem}
\begin{proof} Let $s_1, \ldots, s_{mp}$ be distinct, and let $V(s) = V(s_1, \ldots, s_{mp})$ denote their Vandermonde determinant. Via \autoref{lem:local-index-is-a1-local-degree-wronski} the degree of the Wronski is precisely
\begin{align*}
    \deg^{\A^1} \Wr &= \sum_{W\in Z(\Wr)} \deg^{\A^1}_W \Wr = \left\langle V(s) \cdot  \left( \prod_{i=1}^{m-1} i! \right)^{mp}\right\rangle \sum_{W \in Z(\sigma)} \ind_W \sigma \\
    &= \left\langle V(s) \cdot  \left( \prod_{i=1}^{m-1} i! \right)^{mp}\right\rangle e(\mathcal{V},\sigma).
\end{align*}
By \autoref{lem:Levine-euler-class}, the Euler class is a multiple of $\mathbb{H}$, and since we know the rank of the bilinear form $\deg^{\A^1}\Wr$ in the case where $m$ and $p$ are both even via the classical computation of Schubert, we can determine which integer multiple of the hyperbolic element it must be.
\end{proof}

This global count unifies the real and complex degrees of the Wronski into one computation in these parities --- that is, we recover the complex degree by taking the rank of this form, and the real degree by taking the signature. Contained within the local degree of the Wronski is further geometric information, which we can now explore.

\section{A formula for the local index}\label{sec:local-index}

In this section we will provide a formula for the local degree $\deg_W^{\A^1} \Wr$, when the Wronski has a simple root at the point $W$. To parametrize an affine open cell around $W$, we first fix a basis $e_1, \ldots, e_{m+p}$ of $k_{m+p-1}[t]$ so that $W = \spn\left\{ e_{p+1}, \ldots, e_{m+p} \right\}$. Let $\phi_k$ denote the dual basis element to $e_k$. We may then rewrite the covectors $\sigma_{\ell,1}, \ldots, \sigma_{\ell,m}$ in this dual basis. That is, for any $1\le j\le m$, we write
\begin{equation}
\label{eqn:lambda-in-cobasis}
\begin{aligned}
    \sigma_{\ell,j} := \sum_{k=1}^{m+p} e_k^{(j-1)}(s_\ell) \phi_k,
\end{aligned}
\end{equation}
It is easy to see by acting on $e_k$ by $\sigma_{\ell,j}$, that $e_k^{(j-1)}(s_\ell)$ will be the coefficient on $\phi_k$. By \autoref{eqn:osculating-flag-F}, we have that $\spn \left\{ \sigma_{\ell,1}, \ldots, \sigma_{\ell,m} \right\} = F_m(s_\ell)$, thus by a forgivable abuse of notation we refer to the matrix of coefficients of these vectors as $F_m(s_\ell)$:
\begin{equation}
\label{eqn:B-ell-definition}
\begin{aligned}
    F_m(s_\ell) &=  \begin{pmatrix} e_1(s_\ell) &  e_1'(s_\ell)  & \cdots & e_1^{(m-1)}(s_\ell) \\
            e_2(s_\ell) & e_2'(s_\ell) & \cdots &  e_2^{(m-1)}(s_\ell)\\
            \vdots & \vdots & \ddots & \vdots \\
           e_{m+p}(s_\ell)  & e_{m+p}'(s_\ell)  & \cdots & e_{m+p}^{(m-1)}(s_\ell) \end{pmatrix} = \left(\begin{array}{l | l | l} \text{coeffs of } \sigma_{\ell,1}  & \cdots & \text{coeffs of } \sigma_{\ell,m}  \end{array}\right).
\end{aligned}
\end{equation}
We will define the following notation to identify a \textit{distinguished minor} of this matrix. Namely we want to take minors consisting of all the bottom $m$ rows except one, and one row from higher in the matrix. Explicitly, let $1\le \gamma \le m$ and $1\le k\le p$. Then we denote by $\alpha(\gamma,\kappa)$ the multiindex
\begin{align*}
    \alpha(\gamma,k) := \left\{ k, p+1, \ldots, p+\gamma-1, p+\gamma + 1, \ldots, p+m \right\}.
\end{align*}
In particular this gives us $z_{\alpha(\gamma,k)}(F_m(s_\ell))$,which is the $\alpha(\gamma,k)$th Pl\"{u}cker coordinate of $F_m(s_\ell)$:
\begin{align*}
    z_{\alpha(\gamma,k)}(F_m(s_\ell)) &= \det\begin{pmatrix} e_k(s_\ell) &  e_k'(s_\ell)  & \cdots & e_k^{(m-1)}(s_\ell) \\
    e_{p+1}(s_\ell) &  e_{p+1}'(s_\ell)  & \cdots & e_{p+1}^{(m-1)}(s_\ell) \\
    e_{p+2}(s_\ell) &  e_{p+2}'(s_\ell)  & \cdots & e_{p+2}^{(m-1)}(s_\ell) \\
    \vdots & \vdots &  \ddots  & \vdots \\
    e_{p+\gamma-1}(s_\ell) &  e_{p+\gamma-1}'(s_\ell)  & \cdots & e_{p+\gamma-1}^{(m-1)}(s_\ell) \\
    e_{p+\gamma+1}(s_\ell) &  e_{p+\gamma+1}'(s_\ell)  & \cdots & e_{p+\gamma+1}^{(m-1)}(s_\ell) \\
    \vdots  & \vdots  & \ddots  & \vdots \\
e_{m+p}(s_\ell) &  e_{m+p}'(s_\ell)  & \cdots & e_{m+p}^{(m-1)}(s_\ell) 
\end{pmatrix} \\
&= \Wr\left(e_k, e_{p+1}, \ldots, \widehat{e_{p+\gamma}}, \ldots, e_{m+p}\right)(s_\ell).
\end{align*}
We recall that the fiber of the Wronski over $\prod_{i=1}^{mp}(t-s_i)$ counts the number of $m$-planes meeting $E_p(s_1), \ldots, E_p(s_{mp})$ non-trivially. Here we can state a new geometric interpretation for the local index of the Wronski --- namely it picks up a determinantal relation between distinguished Pl\"{u}cker coordinates of the planes $F_m(s_i)$ (under duality these can be considered as distinguished Pl\"{u}cker coordinates of the $E_p(s_i)$'s). We remark that while our computation of the global degree of the Wronski only held when $m$ and $p$ were both even, the following result holds in all parities and over arbitrary fields, subject to the ongoing assumption that $\left( m+p-1 \right)!^{-1}\in k$.

\begin{theorem}\label{thm:formula-local-index} Let $W$ be a simple preimage of the Wronski in the fiber $\Wr^{-1} \left( \prod_{i=1}^{mp}(t-s_i) \right)$, and let $e_1, \ldots, e_{mp}$ be a basis chosen so that $W = \spn\left\{ e_{p+1}, \ldots, e_{m+p} \right\}$. Then we have that
\begin{align*}
    \deg_W^{\A^1} \Wr &= \gw{ C \cdot \det \mathcal{B} },
\end{align*}
where $C$ is the global constant
\begin{align*}
    C = V(s_1, \ldots, s_{mp}) \left( \prod_{i=1}^{m-1} i! \right)^{mp} (-1)^{m(m-1)p/2},
\end{align*}
where $V(s_1, \ldots, s_{mp})$ is the Vandermonde determinant of the $s_i$'s, and $\mathcal{B}$ is the $mp \times mp$-matrix defined by
\begin{align*}
    \mathcal{B} &= 
    \begin{pmatrix}
    z_{\alpha(1,1)}(F_m(s_1)) & z_{\alpha(1,1)}(F_m(s_2)) & \cdots & z_{\alpha(1,1)}(F(s_{mp})) \\
       z_{\alpha(1,2)}(F_m(s_1)) & z_{\alpha(1,2)}(F_m(s_2)) & \cdots & z_{\alpha(1,2)}(F(s_{mp})) \\
        \vdots & \vdots & \ddots & \vdots \\
       z_{\alpha(1,p)}(F_m(s_1)) &z_{\alpha(1,p)}(F_m(s_2)) & \cdots & z_{\alpha(1,p)}(F(s_{mp})) \\
       z_{\alpha(2,1)}(F_m(s_1)) & z_{\alpha(2,1)}(F_m(s_2)) & \cdots & z_{\alpha(2,1)}(F(s_{mp})) \\
        \vdots & \vdots & \ddots & \vdots \\
        z_{\alpha(2,p)}(F_m(s_1)) & z_{\alpha(2,p)}(F_m(s_2)) & \cdots & z_{\alpha(2,p)}(F(s_{mp})) \\
        \vdots & \vdots & \ddots & \vdots \\
         z_{\alpha(m,p)}(F_m(s_1)) & z_{\alpha(m,p)}(F_m(s_2)) & \cdots & z_{\alpha(m,p)}(F(s_{mp}))
 \end{pmatrix},
\end{align*}
where these Pl\"{u}cker coordinates are written in the basis $\left\{ \phi_i \right\}$.
\end{theorem}
\begin{proof} Since $Z(\sigma) = Z(\Wr)$, we may suppose that $\sigma$ has a simple zero at the top point $W = \spn\left\{ e_{p+1}, \ldots, e_{m+p} \right\} \in \Gr_k(m,m+p)$, and rewrite the covectors of $\sigma$ in the associated cobasis, as in \autoref{eqn:lambda-in-cobasis}. Then we have an affine coordinate chart $U$ around $W$, and we can trivialize $\V$ over $U$ by direct sums of $\til{\phi}_{p+1} \wedge \cdots \wedge \til{\phi}_{m+p}$
We then obtain functions $F_1, \ldots, F_{mp}$ on $U$ defined by
\begin{equation}\label{eqn:F-ell-definition}
\begin{aligned}
    \wedge_{j=1}^m \sigma_{\ell,j} = F_\ell \cdot \til{\phi}_{p+1} \wedge \cdots \wedge \til{\phi}_{m+p}.
\end{aligned}
\end{equation}
The $F_i$'s are local representations of $\sigma$ in the chart $U$, centered around $W$. As $W$ is a simple zero, then in order to compute $\ind_W \sigma$ it suffices to compute the partial derivatives of the functions $F_i$ at the origin of $U$ (which is the point $W = e_{p+1} \wedge \cdots \wedge e_{m+p}$). By the definition of the moving basis in \autoref{eqn:moving-basis}, we have a change of basis formula\footnote{By allowing $\phi_i$ to act on $\til{e}_j$, we get the coefficient of $\til{\phi}_j$ in $\phi_i$.}
\begin{equation}\label{eqn:change-of-cobasis-formula}
\begin{aligned}
    \phi_k = \begin{cases} \til{\phi}_k + \sum_{n=1}^{m} x_{n,k}\til{\phi}_{p+n} & 1 \le k \le p \\ \til{\phi}_k & p+1 \le k \le m+p. \end{cases}
\end{aligned}
\end{equation}
For any fixed $\ell$, we may then write
\begin{equation}\label{eqn:change-of-basis-lin-forms}
\begin{aligned}
    \bigwedge_{j=1}^m \sigma_{\ell,j} &= \bigwedge_{j=1}^m\left( \sum_{k=1}^{m+p} e_k^{(j-1)}(s_\ell) \phi_k \right) \\
    &= \bigwedge_{j=1}^m \left( \sum_{k=1}^p e_k^{(j-1)}(s_\ell) \left( \til{\phi}_k + \sum_{n=1}^m x_{n,k}\til{\phi}_{p+n}  \right) + \sum_{q=p+1}^{m+p} e_q^{(i-1)}(s_\ell) \til{\phi}_q   \right).
\end{aligned}
\end{equation}
Since we will be evaluating this at $e_p \wedge \cdots \wedge e_{m+p-1}$ we only need to worry about terms which are of the form $\til{\phi}_{p+1}\wedge\cdots\wedge \til{\phi}_{m+p-1}$. In particular we can forget about the $\til{\phi}_k$ terms for $1\le k\le p$, and we obtain
\begingroup
\allowdisplaybreaks
\begin{equation}\label{eqn:change-basis-ind-full}
\begin{aligned}
    &\bigwedge_{j=1}^m \left( \sum_{k=1}^p e_k^{(j-1)}(s_\ell)\left(\sum_{n=1}^m x_{n,k}\til{\phi}_{p+n}  \right) + \sum_{k=p+1}^{m+p} e_k^{(i-1)}(s_\ell) \til{\phi}_k   \right) \\
    =& \bigwedge_{j=1}^m \left( \sum_{k=1}^p \sum_{n=p+1}^{m+p}e_k^{(j-1)}(s_\ell) x_{n-p,k} \til{\phi}_n + \sum_{q=p+1}^{m+p} e_q^{(j-1)}(s_\ell) \til{\phi}_q\right) \\
    &= \bigwedge_{j=1}^m \left( \sum_{n=p+1}^{m+p} \left( e_n^{(j-1)}(s_\ell) + \sum_{k=1}^{p} e_k^{(j-1)}(s_\ell) x_{n-p,k} \right) \til{\phi}_n \right) \\
    =& \det (\mathcal{C})\cdot \til{\phi}_{p+1}\wedge \cdots \wedge \til{\phi}_{m+p},
\end{aligned}
\end{equation}
\endgroup
where $\mathcal{C}_{j,\gamma}$ is the coefficient on $\til{\phi}_{p+\gamma}$ in the $j$th exterior power above. Explicitly,
\begin{align*}
    \mathcal{C}_{j,\gamma} = e_{p+\gamma}^{(j-1)}(s_\ell) + \sum_{k=1}^{p} e_k^{(j-1)}(s_\ell) x_{\gamma,k}.
\end{align*}
Since we will evaluate partials at the origin, we only need to pick out linear terms in the $x_{\gamma,k}$'s, so we can forget higher order terms as well as constant terms. Thus, we see that
\begin{align*}
    \det(\mathcal{C}) &= \sum_{\sigma \in S_m} \sgn(\sigma) \prod_{\gamma=1}^{m} \left( e_{p+\gamma}^{(\sigma(\gamma)-1)}(s_\ell) + \sum_{k=1}^{p} e_k^{(\sigma(\gamma)-1)}(s_\ell) x_{\gamma,k} \right).
\end{align*}
For a fixed $x_{\gamma,k}$ the constant coefficient on $x_{\gamma,k}$ is
\begin{equation}\label{eqn:partials-of-local-representation-of-sigma}
\begin{aligned}
    \left.\frac{\partial F_{\ell}}{\partial x_{\gamma,k}}\right|_0 &= \sum_{\sigma \in S_m} \sgn(\sigma) e_k^{(\sigma(\gamma)-1)}(s_\ell) \prod_{\substack{1\le a\le m \\ a\ne \gamma}} e_{p+a}^{(\sigma(a)-1)}(s_\ell),
\end{aligned}
\end{equation}
and we can recognize this as a Pl\"{u}cker coordinate!
\begin{align*}
    \left.\frac{\partial F_{\ell}}{\partial x_{\gamma,k}}\right|_0  &= \Wr(e_{p+1}, \ldots, e_{p+\gamma-1}, e_k, e_{p+\gamma+1}, \ldots, e_{m+p})(s_\ell) \\
    &= (-1)^{\gamma-1} \Wr\left(e_k, e_{p+1}, \ldots, \widehat{e_{p+\gamma}}, \ldots, e_{m+p}\right)(s_\ell) \\
    &= (-1)^{\gamma-1} z_{\alpha(\gamma,k)} \left( F_m(s_\ell) \right).
\end{align*}
In particular by \autoref{rmk:xij-from-pluck-coords} we have that $(-1)^{\gamma}z_{\alpha(\gamma,k)}(F_m(s_\ell))$ is the $(\gamma,k)$th affine coordinate of the plane $F_m(s_\ell)$. Varying over all $(\gamma,k)$ and $\ell$, we obtain the local index as
\begin{align*}
    \ind_W \sigma &= \left\langle \det \left( \frac{\del F_\ell}{\del x_{\gamma,k}} \right)_{(\gamma,k),\ell} \right\rangle \\
    &= \left\langle \det (-1)^{\gamma-1}\left( z_{\alpha(\gamma,k)}(F_m(s_\ell)) \right)_{(\gamma,k),\ell} \right\rangle.
\end{align*}
As $\gamma$ and $k$ vary, we can pull a $(-1)^{\gamma-1}$ out of $p$ different rows, where $\gamma$ is varying from $1$ to $m$. So we have to pull out $(-1)^{p \left( \sum_{\gamma=1}^m \gamma-1 \right)} = (-1)^{m(m-1)p/2}$. This is the coefficient on $(-1)$ we are seeing in the constant for $C$. Finally by \autoref{lem:local-index-is-a1-local-degree-wronski} we have that the local degree of the Wronski and the index of $\sigma$ agree up to these Vandermonde constants.
\end{proof}

\begin{realitycheck} In \cite[Proposition 9]{SW}, the authors demonstrated a formula for the local index of an analogous section in the specific case where $m=2$ and $p=n-1$ for $n$ odd. For the section $\sigma = \oplus_{i=1}^{2n-2} \alpha_i \wedge \beta_i$, they expressed $\alpha_i = \sum_j \alpha_{i,j} \phi_j$ and $\beta_i = \sum_j b_{i,j} \phi_j$, and demonstrated that the local index at $W = e_n \wedge e_{n+1}$ is given by (both in their notation and in the notation from this paper):
\begin{align*}
    \ind_W \sigma = \gw{ \det \begin{vmatrix} \cdots & (a_{i,1} b_{i,n+1} - a_{i,n+1} b_{i,1}) & \cdots \\
        & \vdots & \\
        \cdots & (a_{i,j} b_{i,n+1} - a_{i,n+1} b_{i,j}) & \cdots\\
        & \vdots &  \\
        \cdots & (a_{i,n-1} b_{i,n+1} - a_{i,n+1} b_{i,n-1}) & \cdots \\
        \cdots & (a_{i,n} b_{i,1} - a_{i,1} b_{i,n}) & \cdots \\\
        & \vdots &  \\
       \cdots & (a_{i,n} b_{i,j} - a_{i,j} b_{i,n}) & \cdots\\
        & \vdots &  \\
        \cdots & (a_{i,n} b_{i,n-1} - a_{i,n-1} b_{i,n}) & \cdots
       \end{vmatrix} }
\end{align*}
We note that each of the entries in the $i$th column of this matrix $\mathcal{B}$ is obtained by taking the matrix $\begin{pmatrix} a_{i,n} &  b_{i,n}  \\  a_{i,n+1}& b_{i,n+1} \end{pmatrix}$, swapping out a row for something suitable (as in our construction above), and then taking a determinant. Rewriting this local index in the notation from this paper, we can see that it takes the following form:
\begin{align*}
    \ind_W \sigma &= 
       \gw{ \det \begin{vmatrix} \cdots & z_{\alpha(1,1)}(F_2(s_i)) & \cdots \\
        & \vdots & \\
        \cdots & z_{\alpha(1,j)}(F_2(s_i) & \cdots\\
        & \vdots &  \\
        \cdots & z_{\alpha(1,n-1)}(F_2(s_i)) & \cdots \\
        \cdots & (-1) z_{\alpha(2,1)}(F_2(s_i)) & \cdots \\\
        & \vdots &  \\
       \cdots & (-1) z_{\alpha(2,j)}(F_2(s_i)) & \cdots\\
        & \vdots &  \\
        \cdots & (-1) z_{\alpha(2,n-1)}(F_2(s_i)) & \cdots
       \end{vmatrix} }.
\end{align*}
\end{realitycheck}

\subsection{Maximally inflected curves}
Given an $m$-plane $W = \spn\left\{ f_1, \ldots, f_m \right\}$ with Wronskian $\Wr(f_1, \ldots, f_m)(t) = \prod_{i=1}^{mp}(t-s_i)$, we can consider it as a rational curve $\phi : \P^1 \to \P^{m-1}$, given by mapping $t\mapsto [f_1(t): \ldots : f_m(t)]$. The statement that the Wronski vanishes at $s_i$ is equivalent to the statement that the vectors $\phi(s), \phi'(s), \ldots, \phi^{m-1}(s)$ do not span $\P^{m-1}$ at time $t=s_i$ (c.f. \cite{Sottile,maximally-inflected}). Equivalently one says that the curve \textit{ramifies} or \textit{inflects} at time $s_i$. The degree of the Wronski then admits another interpretation: it counts how many rational curves of degree $\le m+p-1$ have \textit{prescribed inflection} at times $t=s_1, \ldots, s_{mp}$. With our refined local index in hand, we can ask the following question: \textit{How does the local degree} $\deg_W^{\A^1} \Wr$ \textit{of the Wronskian relate to the topology (or geometry) of the associated rational curve $\phi$?}

We don't claim any general answer to this question. Indeed studying topological constraints on inflected curves is a difficult problem in general. In the case when $m=3$ and $p=1,2,3$, we are looking at planar cubics, quartics, and quintics, respectively. Kharlamov and Sottile \cite{maximally-inflected} have studied real inflection data in this setting (by the Shapiro--Shapiro conjecture, when the inflection points are real, the rational curve will be real as well). We can present some very preliminary observations that tie our local degree to their work.

In the case of quartics, there are five different quartics with six flexes (this five is the complex degree of the Wronski whose domain is $\Gr_\C(2,5)$). The graphs of these, pulled from \cite{maximally-inflected}, are included below.\footnote{One remarks that the three leftmost curves have two flexes at the point of self-intersection. This is purely an accident, due to the symmetry on the projective line of the prescribed flex points. In general we shouldn't expect this to happen, and our computations are not impacted by this coincidence.} While the curves look topologically distinct due to the nodal singularities, it is perhaps more telling to look at the number of isolated points (real ordinary points with complex conjugate tangent directions).

\begin{figure}
    \begin{tabular}{ r  c c c c c}
    Curve & \includegraphics[width=0.12\linewidth]{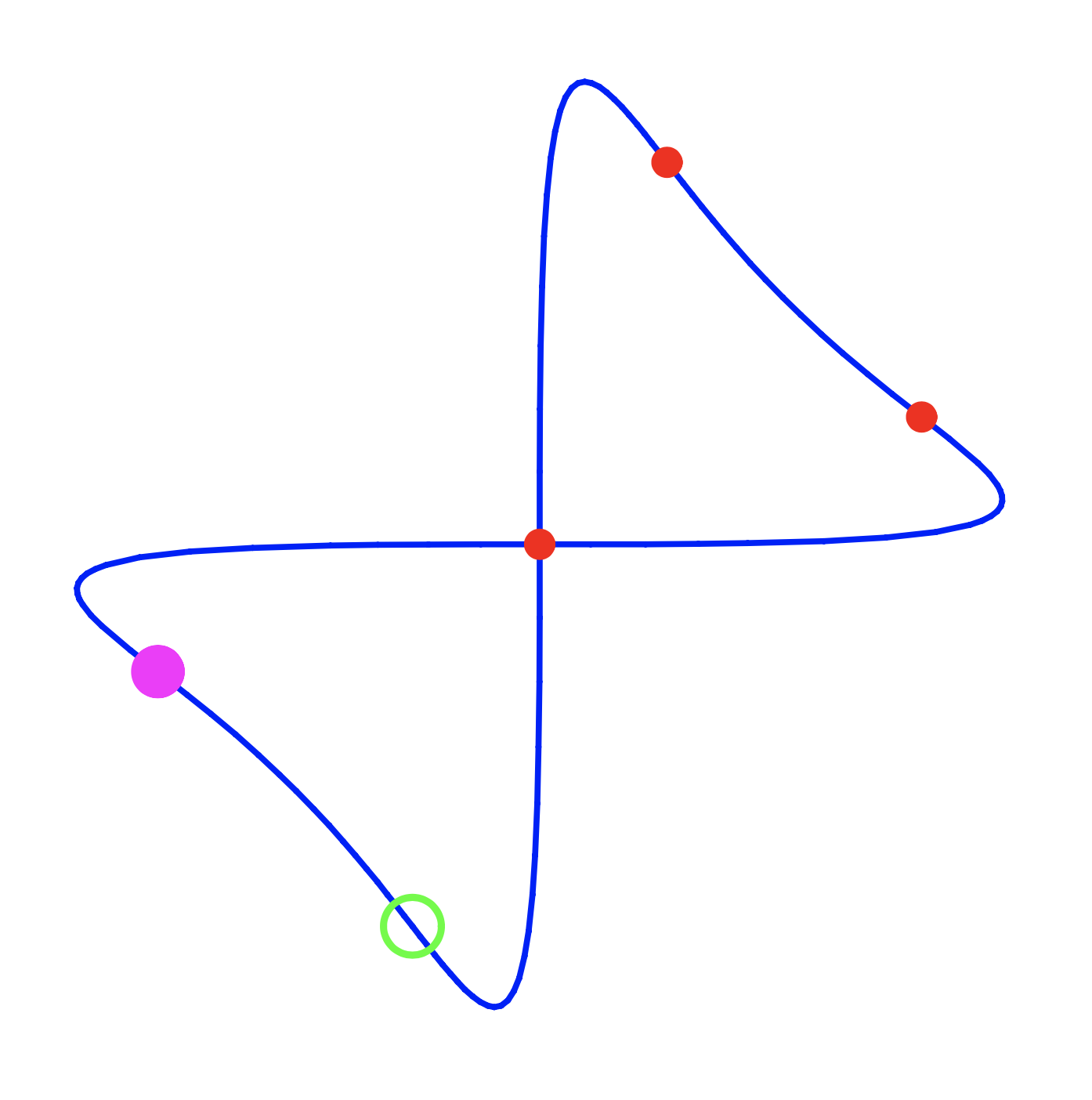} & \includegraphics[width=0.12\linewidth]{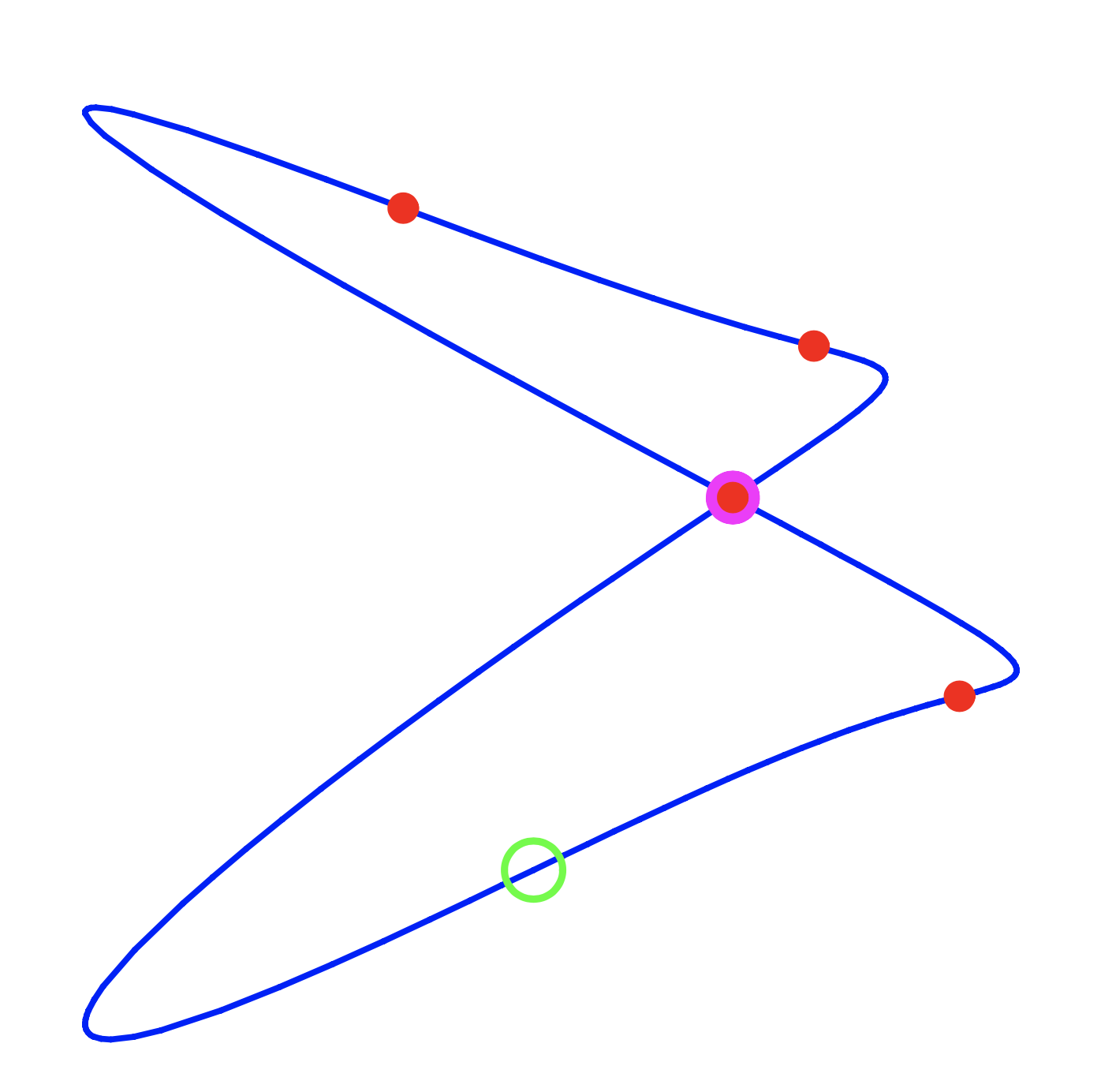} & \includegraphics[width=0.12\linewidth]{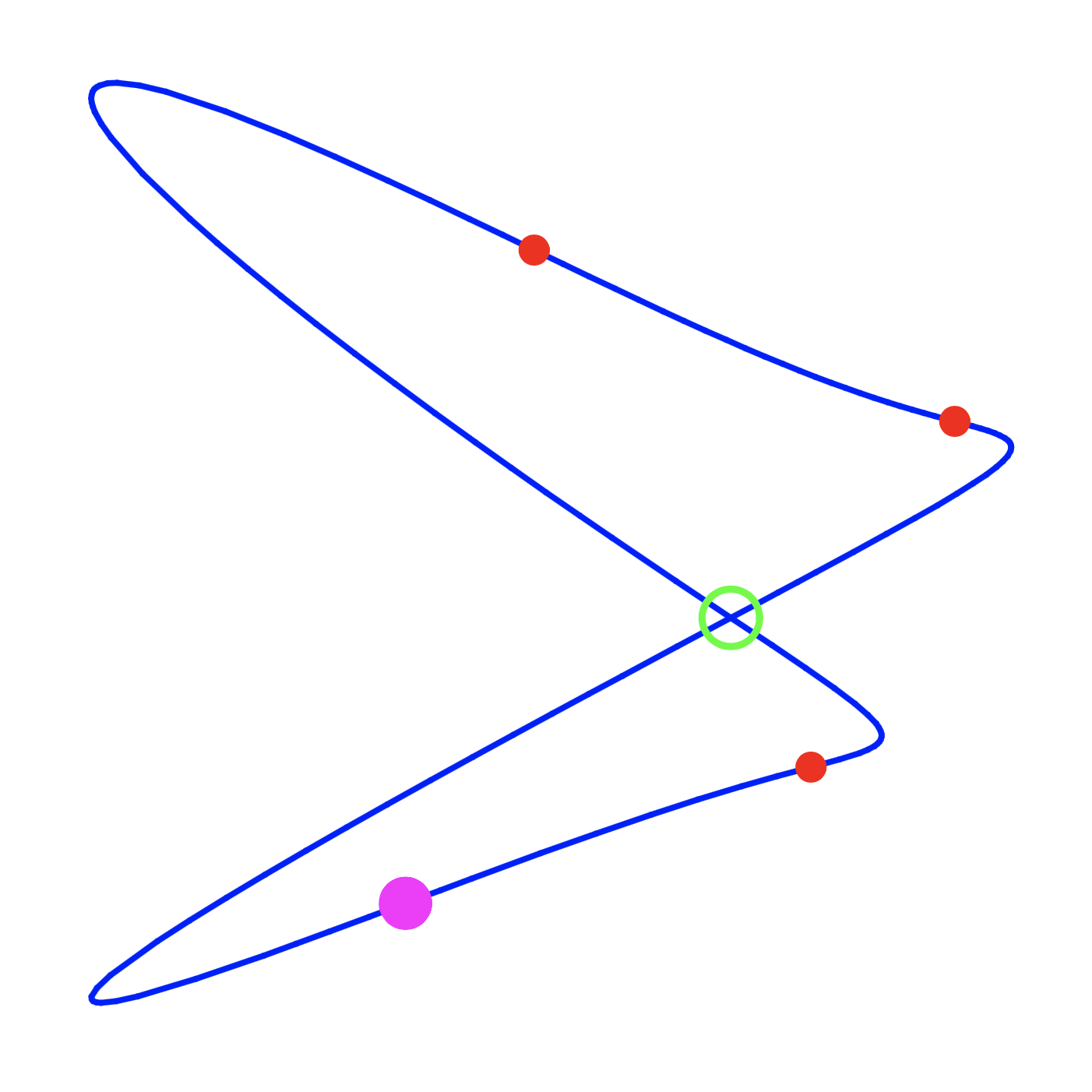} & \includegraphics[width=0.12\linewidth]{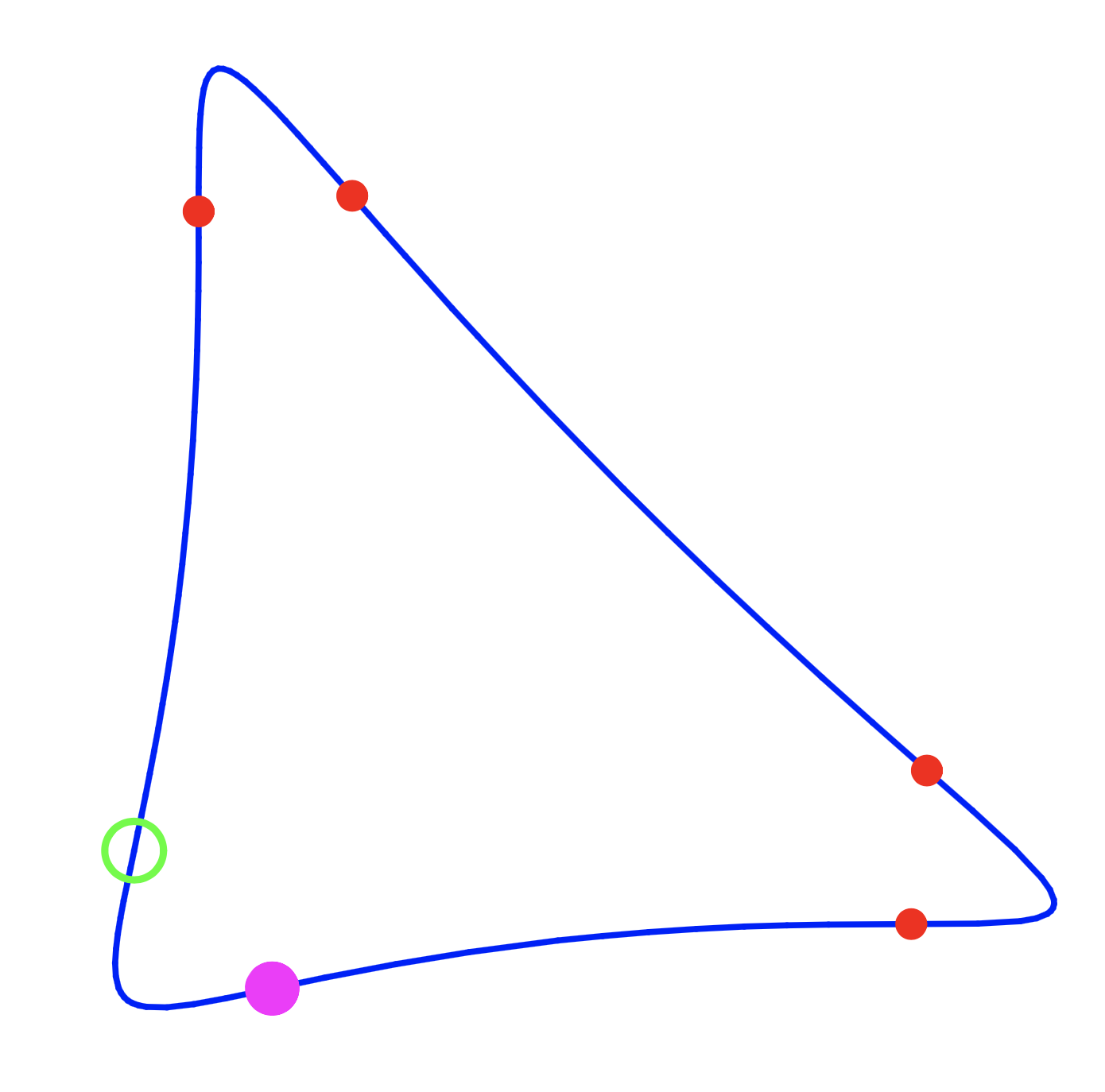} & \includegraphics[width=0.12\linewidth]{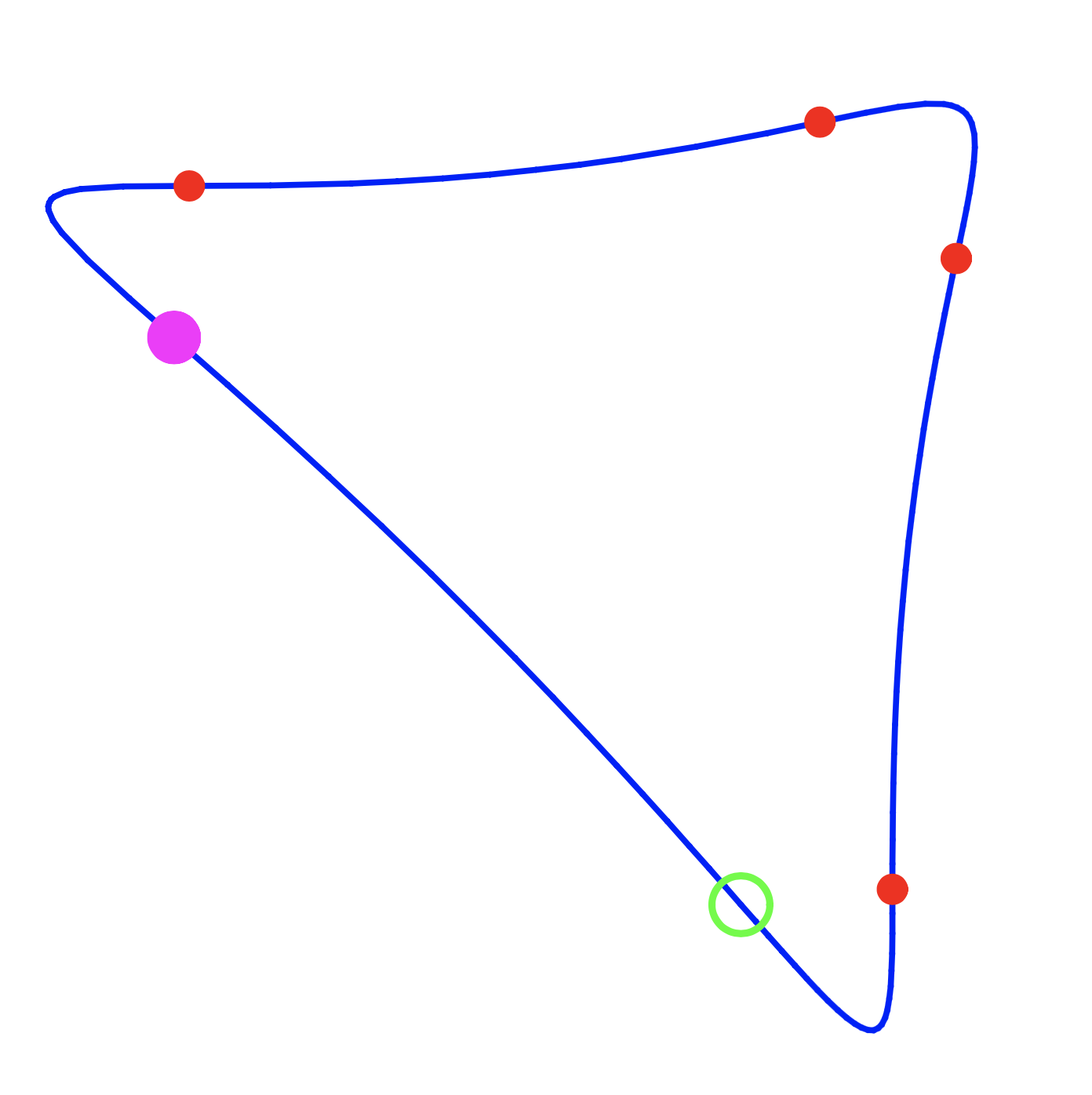} \\
    & &  \\
    \# isolated points & 2 & 2 & 2 & 3 & 3 \\
    Local $\A^1$-index over $\mathbb{R}$ & $+1$ & $+1$ & $+1$ & $-1$ & $-1$.
    \end{tabular}
    \centering
    \caption{Maximally inflected real quartics, \cite[p.~23]{maximally-inflected}}
\end{figure}
Welschinger \cite{Welschinger} remarked that $(-1)^{\#I}$ is a revealing invariant to consider for planar curves, where $I$ is the set of isolated points. Kass--Levine--Solomon--Wickelgren have extended this to define an arithmetic Welschinger invariant valued in $\GW(k)$ \cite{KLSW} (see also \cite{Levine-Welschinger} and \cite{Pauli-expos}). We may compute that Welschinger's original invariant agrees with the local index of $\sigma$ following the formula in \autoref{thm:formula-local-index}.
\begin{corollary}\label{cor:quartics} When $(f_1:f_2:f_3)$ defines a real planar quartic, we have that
\begin{align*}
    \sgn\deg_{\spn\left\{ f_1,f_2,f_3 \right\}}^{\A^1} \Wr = (-1)^{\# I},
\end{align*}
where $(-1)^{\#I}$ is Welschinger's invariant.
\end{corollary}
It is possible that the arithmetic Welschinger invariant provides a local $\A^1$-degree for the Wronski, which could potentially shine light on the classification of maximally inflected curves in higher degrees and higher dimensions. We plan to explore this idea in greater detail in a future paper.

\printbibliography
\end{document}